\documentclass[11pt]{article}

\usepackage[english]{babel}
\usepackage[utf8x]{inputenc}
\usepackage[T1]{fontenc}
\usepackage[letterpaper,top=1in,bottom=1in,left=0.9in,right=0.9in,marginparwidth=1.75cm]{geometry}
\usepackage{dsfont}
\usepackage{amsmath, amssymb, amsthm, verbatim,enumerate,bbm}
\usepackage{indentfirst, bbm}
\usepackage{appendix}
\usepackage{enumerate}
\usepackage{verbatim}
\usepackage[colorinlistoftodos]{todonotes}
\usepackage[colorlinks=true, allcolors=blue]{hyperref}
\usepackage[nameinlink,capitalise,noabbrev]{cleveref}
\crefname{appsec}{Appendix}{Appendices}
\creflabelformat{enumi}{#2textup{#1}#3}

\allowdisplaybreaks
\makeatletter
\newtheorem{theorem}{Theorem}[section]

\newtheorem*{namedtheorem}{\theoremname}
\newcommand{\theoremname}{testing}

\newtheorem{lemma}[theorem]{Lemma}

\newtheorem{proposition}[theorem]{Proposition}

\newtheorem*{question*}{Question}

\theoremstyle{definition}
\newtheorem{definition}[theorem]{Definition}

\newtheorem{remark}[theorem]{Remark}

\theoremstyle{plain}

\makeatother

\title{Approximate Spielman-Teng theorems for the least singular value of random combinatorial matrices}
\author{
Vishesh Jain\thanks{Massachusetts Institute of Technology. Department of Mathematics. Email: {\tt visheshj@mit.edu}. 
}}
\date{}
\DeclareMathOperator{\LCD}{LCD}
\DeclareMathOperator\sign{sgn}
\DeclareMathOperator\Proj{Proj}

\begin{document}
\maketitle
\global\long\def\R{\mathbb{R}}

\global\long\def\S{\mathbb{S}}

\global\long\def\Z{\mathbb{Z}}

\global\long\def\C{\mathbb{C}}

\global\long\def\Q{\mathbb{Q}}

\global\long\def\N{\mathbb{N}}

\global\long\def\P{\mathbb{P}}

\global\long\def\F{\mathbb{F}}

\global\long\def\U{\mathcal{U}}

\global\long\def\V{\mathcal{V}}

\global\long\def\E{\mathbb{E}}

\global\long\def\Ev{\mathscr{Rk}}

\global\long\def\Dg{\mathscr{D}}

\global\long\def\Ndg{\mathscr{ND}}

\global\long\def\Rv{\mathcal{R}}

\global\long\def\Gv{\mathscr{Null}}

\global\long\def\Hv{\mathscr{Orth}}

\global\long\def\Supp{{\bf Supp}}

\global\long\def\Sv{\mathscr{Spt}}

\global\long\def\ring{\mathfrak{R}}

\global\long\def\1{\mathbbm{1}}

\global\long\def\Bad{{\boldsymbol {B}}}

\global\long\def\supp{{\bf supp}}

\global\long\def\A{\mathcal{A}}

\global\long\def\L{\mathcal{L}}

\global\long\def\dist{\text{dist}}

\global\long\def\QQ{\mathcal{Q}}

\global\long\def\BB{\mathcal{B}}

\global\long\def\EE{\mathcal{E}}

\begin{abstract}
An approximate Spielman-Teng theorem for the least singular value $s_n(M_n)$ of a random $n\times n$ square matrix $M_n$ is a statement of the following form: there exist constants $C,c >0$ such that for all $\eta \geq 0$, $\Pr(s_n(M_n) \leq \eta) \lesssim n^{C}\eta + \exp(-n^{c})$. The goal of this paper is to develop a simple and novel framework for proving such results for discrete random matrices. As an application, we prove an approximate Spielman-Teng theorem for $\{0,1\}$-valued matrices, each of whose rows is an independent vector with exactly $n/2$ zero components. This improves on previous work of Nguyen and Vu, and is the first such result in a `truly combinatorial' setting.      
\end{abstract}

\section{Introduction}
Let $M_n$ be an $n\times n$ real matrix. Its \emph{singular values}, denoted by $s_k(M_n)$ for $k\in [n]$ are the eigenvalues of $\sqrt{M_{n}^{T}M_{n}}$ arranged in non-decreasing order. Of particular interest are the largest and smallest singular values, which have the following variational characterizations:
$$s_1(M_{n}):= \sup_{\boldsymbol{x}\in \S^{n-1}}\|M_{n}\boldsymbol{x}\|_{2};$$
$$s_n(M_{n}):= \inf_{\boldsymbol{x}\in \S^{n-1}}\|M_{n}\boldsymbol{x}\|_{2},$$
where $\|\cdot \|_{2}$ denotes the Euclidean norm on $\R^{n}$, and $\S^{n-1}$ denotes the $n-1$ dimensional Euclidean sphere in $\R^{n}$. 

The study of the \emph{non-limiting} or \emph{non-asymptotic} behavior of the largest and smallest singular values of random matrices plays a crucial role in diverse areas of mathematics -- such as applied linear algebra, computer science, statistics, and asymptotic geometric analysis -- in addition to often being a key ingredient in proving other results in random matrix theory, for instance the \emph{circular law} (which is the non-Hermitian counterpart of the classical \emph{semicircle law} of Wigner) and \emph{delocalization} properties of eigenvectors. We refer the reader to the the surveys \cite{nguyen2013small, rudelson2010non, vershynin2010introduction} and the books \cite{tao2012topics, tao2006additive} for a detailed account of the development of the area.  

The behavior of the largest singular value of random matrices with independent entries is relatively well-understood. Lata\l{}a \cite{latala2005some} showed that if the entries of $M_n$ have mean $0$ and uniformly bounded fourth moment,
then with high probability,
$$s_1(M_n) = O(n^{1/2});$$
for i.i.d. entries with mean $0$, variance $1$, and uniformly bounded fourth moment, it was already known much earlier \cite{bai1988note, yin1988limit} that with high probability,
$$s_1(M_n) = \Theta(n^{1/2}).$$

On the other hand, the study of the behavior of the smallest singular value has proved to be much harder. For an overview of the history of this problem for matrices with i.i.d. entries, we refer the reader to \cite{rudelson2008littlewood}; here, we only briefly summarize a few developments. For random matrices with i.i.d. standard Gaussian entries, it was proved by Edelman \cite{edelman1988eigenvalues} that
\begin{align}
\label{eqn:edelman-bound}
\Pr\left(s_n(M_n) \leq \varepsilon n^{-1/2}\right) \sim \varepsilon,
\end{align}
thereby confirming (in a very strong form) a conjecture of Smale, and a speculation of von Neumann and Goldstine. In connection with their work on \emph{smoothed analysis}, Spielman and Teng \cite{spielman2004smoothed} conjectured that \cref{eqn:edelman-bound} should also hold for random Rademacher matrices (i.e. each entry is independently $\pm 1$ with equal probability), up to an additive error of $c^{n}$ (for some $c < 1$) to account for the probability that such a matrix is singular; i.e., they conjectured that
\begin{align}
\label{eqn:spielman-teng-conj-signed}
\Pr\left(s_n(M_n) \leq \varepsilon n^{-1/2}\right)\leq \varepsilon + c^{n}. 
\end{align}
Note that the $\varepsilon = 0$ version of this conjecture asserts that the probability that a random signed matrix is singular is exponentially small; even proving that this probability goes to $0$ as $n \to \infty$ is a non-trivial result due to Koml\'os \cite{komlos1967determinant}, and the exponential bound was only obtained much later by Kahn, Koml\'os and Szemer\'edi \cite{kahn1995probability}.

It was shown in breakthrough works by Rudelson \cite{rudelson2008invertibility} that
$$\Pr\left(s_n(M_n) \leq \varepsilon n^{-3/2}\right) \lesssim \varepsilon + n^{-3/2}$$
for all random matrices $M_n$ with i.i.d. centered subgaussian entries, and by Tao and Vu \cite{tao2009inverse} that for random signed matrices $M_n$, for any $A>0$, there exists $B>0$ such that 
\begin{equation}
\label{eqn:tv-type}
\Pr\left(s_n(M_n) \leq n^{-B}\right) \leq n^{-A}.
\end{equation}
These results have been greatly refined in subsequent remarkable works: Rudelson and Vershynin \cite{rudelson2008littlewood} showed that \cref{eqn:spielman-teng-conj-signed} holds (up to a multiplicative constant) for all random matrices with i.i.d. centered subgaussian entries, Rebrova and Tikhomirov \cite{rebrova2018coverings} proved the same result assuming only that the i.i.d. entries are centered and have variance $1$, and in the special case of random signed matrices, Tikhomirov \cite{tikhomirov2018singularity} proved the same result but with the correct `constant' $c = (1/2 + o_n(1))$.\\
\\
\textbf{Random matrices with dependent entries:} Despite the great progress in the study of random matrices with independent entries, much less is known about the behavior of the least singular value for models of random matrices with non-trivial dependence between entries. Some measure of the difficulty in the study of such models may be obtained by noting that the \emph{symmetric} analog of Koml\'os's classical result (on the asymptotically almost sure invertibility of random Bernoulli matrices) was only proved almost 40 years later (in 2006) by Costello, Tao, and Vu \cite{costello2006random}. Similarly, while the Spielman-Teng conjecture for random signed matrices has been settled up to an overall constant, the current best statement of the same form for random \emph{symmetric} signed matrices $M_n$ is due to Vershynin \cite{vershynin2014invertibility}, who proved that 
\begin{align}
\label{eqn:vershynin-symmetric}
\Pr\left(s_n(M_n) \leq \varepsilon n^{-1/2}\right) \lesssim \varepsilon^{1/9} + e^{-n^c}
\end{align}
for some small constant $c>0$. Motivated by this, we will henceforth refer to a result of the following form as an \emph{approximate Spielman-Teng theorem} for a random matrix $M_n$; these will be the subject of the present work: there exist constants $C,c >0$ such that 
\begin{equation}
    \label{eqn:model-ast}
    \Pr\left(s_n(M_n) \leq \varepsilon n^{-1/2}\right) \lesssim n^{C}\varepsilon + e^{-n^c}.  
\end{equation}

In recent years, motivated by combinatorial applications, the study of such questions for the adjacency matrices of random graphs has attracted a lot of attention, with particular emphasis on graphs or bipartite graphs satisfying various regularity constraints (which translate to constraints on the row/column sums of the matrix). In these settings, even the analogs of Koml\'os's theorem have only very recently been proved -- for $d$-regular digraphs with $n-3 \geq d\geq 3$, this is due to (complementary) work of Cook \cite{cook2017singularity}, Litvak, Lytova, Tikhomirov, Tomczak-Jaegermann
and Youssef \cite{litvak2017adjacency}, and Huang \cite{huang2018invertibility}, whereas for $d$-regular graphs with $n-3 \geq d \geq 3$, this is due to Landon, Sosoe, and Yau \cite{landon2019fixed}, and Huang \cite{huang2018invertibility} (see also the parallel works of M\'esz\'aros \cite{meszaros2018distribution} and Nguyen and Wood \cite{nguyen2018nonsingularity}). Whereas some quantitative control on the least singular value in combinatorial settings has been obtained (see, e.g., \cite{cook2017circular} and \cite{litvak2017smallest}, and also the discussion below regarding \cite{nguyen2012circular}), these bounds are still quite far from approximate Spielman-Teng type results. In fact, prior to the very recent work of the Ferber, Jain, Luh, and Samotij \cite{FJLS2018}, we are not even aware of any `\emph{exponential-type}' bound (by which we mean a bound of the form $\exp(-n^{c})$ for some constant $c>0$) on the singularity probability in combinatorial settings of such nature.     

\subsection{Our results}
Our goal in this work is to establish a novel framework (utilizing the recent approach to the `counting problem in inverse Littlewood--Offord theory' developed by the author, along with Ferber, Luh, and Samotij \cite{FJLS2018}) for proving approximate Spielman-Teng results in the discrete setting in a simple and unified manner. As an illustration of our main techniques (while keeping technicalities to a minimum), we begin by providing a proof of the following theorem which, in our opinion, is much simpler than existing proofs in the literature. 
\begin{theorem}
\label{thm:lsv-iid}
Let $M_n$ denote an $n\times n$ random matrix, each of whose entries is an independent Rademacher random variable. Then, for any $ \eta \geq 2^{-n^{0.0001}}$,
$$\Pr\left(s_n(M_n) \leq \eta \right) \lesssim \eta n^{3/2}.$$
\end{theorem}
\begin{remark}
We have not made any attempt to optimize the constant $0.0001$ or the factor $n^{3/2}$ in the above theorem, choosing instead to keep the exposition simple. We also note that our proof goes through with very minor modifications to yield a similar result for the case when the entries of $M_n$ are i.i.d., with each entry taking on the value $0$ with probability $1-\mu$ and $\pm 1$ with probability $\mu/2$ each, for some fixed constant $\mu \in (0,1]$, thereby providing a simple new proof of (a quantitative improvement of) the main result of Tao and Vu in \cite{tao2009inverse}. On the other hand, as mentioned in the introduction, better and nearly optimal quantitative bounds are already known in this case. 
\end{remark}

Next, we use our general framework, along with certain combinatorial ideas developed in \cite{FJLS2018}, to prove (to the best our knowledge) the first approximate Spielman-Teng theorem in a `truly combinatorial' setting. 
\begin{theorem}
\label{thm:lsv-rreg}
Let $n\in \N$ be even, and let $Q_n$ denote an $n\times n$ random matrix, sampled uniformly from $n\times n$ $\{0,1\}$-valued matrices, each of whose rows sums to $n/2$. Then, for any $\eta \geq 2^{-n^{0.0001}}$,
$$\Pr(s_n(Q_n) \leq \eta )\lesssim \eta n^{2}.$$
\end{theorem}

\begin{remark}
Once again, we have not tried to optimize the constant $0.0001$ or the factor $n^{2}$ in the above theorem. The restriction to row sums being equal to $n/2$ is also made for simplicity; similar ideas may be used to prove a statement like the one above with $n/2$ replaced by some other row sum $s$ satisfying $\epsilon n \leq s \leq (1-\epsilon)n$ for some fixed $\epsilon > 0$.
\end{remark}

The problem of estimating the probability that $Q_n$ is singular was first considered by Nguyen in \cite{nguyen2013singularity} (as a step towards understanding the regular digraph/graph case), where it was shown that, for any constant $C>0$,
$$\Pr(Q_n \text{ is singular}) = O_{C}(n^{-C}).$$
An exponential-type upper bound on this probability was recently provided in \cite{FJLS2018}. The question of obtaining quantitative lower tail bounds on the least singular value of $Q_n$ was considered by Nguyen and Vu in \cite{nguyen2012circular}, where a much weaker bound of the form \cref{eqn:tv-type} was obtained. The goal of that work was to prove a circular law for such matrices; while we do not consider this matter here, we remark that obtaining quantitative lower tail estimates on the least singular value (of perturbed matrices) is a key step in proving circular laws, and we believe that our techniques should extend to that setting as well. We also believe that our techniques (combined with additional combinatorial arguments) should allow one to prove an approximate Spielman-Teng theorem for sufficiently dense random regular digraphs.     

\subsection{Discussion and future work}
We will discuss the main ingredients of our method in detail in the next two sections; here, we make a few general remarks. Our general approach to proving lower tail estimates on the least singular value lies somewhere between the method of Tao and Vu (as developed in \cite{tao2007singularity} and subsequent works), and the method of Rudelson and Vershynin (as developed in \cite{rudelson2008littlewood} and subsequent works). Like Tao and Vu, we reduce to working with integer vectors (as opposed to working with nets on the unit sphere); however, we completely avoid the use of inverse Littlewood-Offord type theorems, choosing instead to work with the simple and quantitatively stronger counting variant developed in \cite{FJLS2018}. On the other hand, like Rudelson and Vershynin, we utilize the key notion of the \emph{Least Common Denominator (LCD)} of a vector. However, while their work requires dividing vectors on the unit sphere into approximate level sets of the LCD and carefully analyzing each piece, we only need to distinguish `large' LCD from `small' LCD. Interestingly, our method provides a view of the LCD as a bridge from the problem of controlling the least singular value to the problem of controlling the singularity probability on a subset of integer vectors. 

In upcoming work, we will build upon the ideas introduced here in a couple of directions. In \cite{jain2019b}, we extend the techniques of \cite{FJLS2018} to prove a counting counterpart for the inverse Littlewood-Offord problem for very general distributions, and use this to provide a simple combinatorial proof of an approximate Spielman-Teng theorem for random matrices with i.i.d. heavy-tailed entries (a Spielman-Teng theorem for such matrices was recently proved by Rebrova and Tikhomirov \cite{rebrova2018coverings}). In \cite{jain2019c}, we further develop the ideas here to prove approximate Spielman-Teng results in the important setting of smoothed analysis i.e. when the random matrix is perturbed by a fixed, polynomially bounded matrix; here, weaker bounds of the form \cref{eqn:tv-type} are known due to Tao and Vu \cite{tao2008random}.\\

\noindent {\bf Organization: }The remainder of this paper is organized as follows. In \cref{sec:outline-iid-proof}, we provide a high-level outline of the proof of \cref{thm:lsv-iid} (the proof of \cref{thm:lsv-rreg} is conceptually quite similar, and we will discuss the necessary changes at the start of \cref{sec:proof-thm-rreg}); in \cref{sec:tools}, we collect some tools and auxiliary results which will be used in the proofs of our main results. Finally, in \cref{sec:proof-thm-iid} and \cref{sec:proof-thm-rreg}, we prove \cref{thm:lsv-iid} and \cref{thm:lsv-rreg} respectively.  
\\

\noindent {\bf Notation: } Throughout the paper, we will omit floors and ceilings when they make no essential difference. For convenience, we will also say `let $p = x$ be a prime', to mean that $p$ is a prime between $x$ and $2x$; again, this makes no difference to our arguments. As is standard, we will use $[n]$ to denote the discrete interval $\{1,\dots,n\}$. We will also use the asymptotic notation $\lesssim, \gtrsim, \ll, \gg$ to denote $O(\cdot), \Omega(\cdot), o(\cdot), \omega(\cdot)$ respectively. All logarithms are natural unless noted otherwise.  
\\

\noindent {\bf Acknowledgements: }I would like to thank Kyle Luh for comments on a preliminary version of this paper, and Jake Lee Wellens for helpful conversations.

\section{Outline of the proof of \cref{thm:lsv-iid}}
\label{sec:outline-iid-proof}
\subsection{The approach of Tao and Vu}
To motivate our proof, we begin by recalling the high-level approach of Tao and Vu from \cite{tao2009inverse}. Let $B>10$ be a large number (depending on $A$) to be chosen later. Then, if $s_n(M_{n}) < n^{-B}$, there must exist a unit vector $\boldsymbol{v}\in \S^{n-1}$ for which 
$$\|M_{n} \boldsymbol{v}\|_{2} < n^{-B}.$$ 
By rounding each coordinate $\boldsymbol{v}$ to the nearest multiple of $n^{-B-2}$, we can find a vector $\tilde{\boldsymbol{v}} \in n^{-B-2}\cdot \Z^{n}$ of magnitude $0.9 \leq \|\tilde{\boldsymbol{v}}\|_{2} \leq 1.1$ such that
$$\|M_{n}\tilde{\boldsymbol{v}}\|_{2} \leq 2n^{-B}.$$ 
Hence, writing $\boldsymbol{w}:=n^{B+2}\tilde{\boldsymbol{v}}$, we can find an integer vector $\boldsymbol{w} \in \Z^{n}$ of magnitude $0.9 n^{B+2} \leq \|\boldsymbol{w}\|_{2} \leq 1.1 n^{B+2}$ such that
$$\|M_{n}\boldsymbol{w}\|_{2} \leq 2n^{2}.$$

Let $\Omega$ be the set of integer vectors $\boldsymbol{w} \in \Z^{n}$ of magnitude $0.9 n^{B+2} \leq \|\boldsymbol{w}\|_{2} \leq 1.1 n^{B+2}$. By the above discussion, it suffices to show that
$$\Pr\left(\exists \boldsymbol{w} \in \Omega\text{ such that } \|M_{n}\boldsymbol{w}\|_{2}\leq 2n^{2}\right) = O_{A}(n^{-A}).$$
In order to show this, Tao and Vu partition the elements of $\Omega$ into three sets, which they analyze using separate arguments. This partition is based on whether or not the vector is `close' to a sufficiently low-dimensional subspace, as well as the following key quantity.
\begin{definition}[Largest atom probability] For an integer vector $\boldsymbol{w}\in \Z^{n}$, we define its largest atom probability to be
$$\rho(\boldsymbol{w}) := \sup_{x\in \Z}\Pr\left(\epsilon_{1}w_{1}+\dots+\epsilon_{n}w_{n} = x\right),$$
where $\epsilon_1,\dots,\epsilon_n$ are  i.i.d. random Rademacher variables. 
\end{definition}

\noindent The partitioning scheme of Tao and Vu is as follows: 
\begin{itemize}
    \item A vector $\boldsymbol{w}\in \Omega$ is \emph{rich} if $\rho(\boldsymbol{w}) \geq n^{-A-10}$ and \emph{poor} otherwise. Let $\Omega_{1}$ be the set of poor $\boldsymbol{w}$'s.
    \item A rich $\boldsymbol{w}$ is \emph{singular} is fewer than $n^{0.2}$ of its coordinates have absolute value $n^{B-10}$ or greater. Let $\Omega_{2}$ be the set of rich and singular $\boldsymbol{w}$'s.
    \item A rich $\boldsymbol{w}$ is \emph{nonsingular} if at least $n^{0.2}$ of its coordinates have absolute value $n^{B-10}$ or greater. Let $\Omega_{3}$ be the set of rich and nonsingular $\boldsymbol{w}$'s.
\end{itemize}
The desired claim follows directly from the estimates below and the union bound. 
\begin{itemize}
    \item (Lemma 7.1 in \cite{tao2009inverse}) $\Pr\left(\exists \boldsymbol{w} \in \Omega_{1} : \|M_{n}\boldsymbol{w}\|_{2}\leq 2n^{2}\right) = o_{A}(n^{-A})$.
    \item (Lemma 7.2 in \cite{tao2009inverse}) $\Pr\left(\exists \boldsymbol{w} \in \Omega_{2} : \|M_{n}\boldsymbol{w}\|_{2}\leq 2n^{2}\right) = o_{A}(n^{-A})$.
    \item (Lemma 7.3 in \cite{tao2009inverse})
    $\Pr\left(\exists \boldsymbol{w} \in \Omega_{3} : \|M_{n}\boldsymbol{w}\|_{2}\leq 2n^{2}\right) = o_{A}(n^{-A})$.
\end{itemize}
The proofs of the first two bullet points above are relatively straightforward and standard, and based on similar proofs in \cite{litvak2005smallest,rudelson2008invertibility}. The main work in \cite{tao2009inverse} is the proof of the third bullet point, which requires the inverse Littlewood-Offord theorems along with additional additive combinatorial arguments.

\subsection{Our approach}
The starting point of our approach is the following simple observation. Let $\Gamma \subseteq \Z^{n}$ be a set of non-zero integer vectors. Then,
\begin{align}
\label{eqn:union-bound-over-ball}
\Pr\left(\exists \boldsymbol{w}\in \Gamma: \|M_{n}\boldsymbol{w}\|_{2}\leq C(n)\sqrt{n}\right) 
&\leq \sum_{\boldsymbol{z} \in \Z^{n}\cap B(0,C(n)\sqrt{n})}\Pr\left(\exists \boldsymbol{w} \in \Gamma: M_{n}\boldsymbol{w} = \boldsymbol{z}\right) \nonumber \\
\leq \left|\Z^{n}\cap B(0,C(n)\sqrt{n})\right|&\cdot \sup_{\boldsymbol{z}\in \Z^{n}\cap B(0,C(n)\sqrt{n})}\Pr\left(\exists \boldsymbol{w} \in \Gamma: M_n \boldsymbol{w} = \boldsymbol{z}\right) \nonumber \\
\leq \left(100C(n)\right)^{n}&\cdot \sup_{\boldsymbol{z}\in \Z^{n}}\Pr\left(\exists \boldsymbol{w} \in \Gamma: M_n \boldsymbol{w} = \boldsymbol{z}\right),
\end{align}
where the first equality uses that $M_{n}\boldsymbol{w}$ is always an integer vector, and the last inequality uses a standard (loose) volumetric estimate on the number of integer points in an $n$-dimensional ball of radius $R$. The second quantity in the last equation i.e. 
\begin{align}
\label{eqn:matrix-atom}
\sup_{\boldsymbol{z}\in \Z^{n}}\Pr\left(\exists \boldsymbol{w} \in \Gamma: M_n \boldsymbol{w} = \boldsymbol{z}\right)
\end{align}
is reminiscent of the singularity problem for random Rademacher matrices, which corresponds to the case when $\Gamma = \Z^{n}\setminus \{0\}$ and the supremum is replaced simply by $\boldsymbol{z}=0$. The bounds on the singularity problem coming from either inverse Littlewood-Offord theory \cite{tao2009inverse} or its counting variant \cite{FJLS2018} show that for a suitable set of vectors of `intermediate' largest atom probability, one may bound the quantity in \cref{eqn:matrix-atom} by $O(n^{- c n})$ for some (small) absolute constant $c>0$ (see also \cref{prop:eliminate-smallLCD-integer}). Hence, for $C(n) = o(n^{c})$, the quantity on the right hand side of \cref{eqn:union-bound-over-ball} is $(o(1))^{n}$. 

Since the set of vectors of `intermediate' largest atom probability mentioned above correspond, in a sense, to `rich, nonsingular' vectors, one may hope to use a similar decomposition of integer vectors as Tao and Vu to complete the proof. However, 
one runs into the immediate obstacle that the discussion in the above paragraph only holds for $C(n) = o(n^{c})$
, whereas the reduction to integer vectors in \cite{tao2009inverse} requires one to be able to work with $C(n) = \Omega(n^{3/2})$. Note that this reduction, as stated, is clearly wasteful; by using the fact (\cref{prop:bound-operator-norm}) that, except with exponentially small probability, $\|M_{n}\| = O(\sqrt{n})$, one is able to reduce the consideration to $C(n) = O(\sqrt{n})$, which turns out to be just out of reach.

However, this loss is because we are using the worst-case estimate that the closest vector $\boldsymbol{w}\in n^{-B-1}\cdot \Z^{n}$ to a given vector $\boldsymbol{v} \in \S^{n-1}$ satisfies $\|\boldsymbol{w}-\boldsymbol{v}\|_2 \leq n^{-B-1/2}$. To overcome this obstacle, we will use the connection between largest atom probability and Diophanine approximation (as captured by the Least Common Denominator (LCD)) developed in \cite{rudelson2008littlewood}. In particular, we will use the fact (\cref{prop:LCD-controls-sbp}) that vectors $\boldsymbol{v}\in \S^{n-1}$ for which this worst-case estimate is `close' to being true have high LCD, and hence, are necessarily `poor'; in other words, for `rich' vectors, we gain sufficiently over the worst-case estimate (\cref{lemma:reduction-to-integer}) for the above strategy to be effective.   

\section{Tools and auxiliary results}
\label{sec:tools}
\subsection{Anti-concentration and the LCD} 
In this subsection, we record the definition of the LCD of a vector and its connection to the classical L\'evy concentration function, as developed in \cite{rudelson2008littlewood}.
\begin{definition}
\label{def:Levy-conc}
The L\'evy concentration function of a random variable $X$ at scale $\delta \geq 0$  is defined
as 
\[
\L(X,\delta):=\sup_{r\in\R}\Pr\left(|X-r|\leq\delta\right).
\]
\end{definition}
\begin{definition}[Least Common Denominator (LCD)]
\label{def:LCD}
For $\gamma\in(0,1)$ and $\alpha>0$, and for a non-zero vector $\boldsymbol{a}\in \R^{n}$, define 
\[
\LCD_{\gamma,\alpha}(\boldsymbol{a}):=\inf\left\{ \theta>0:\dist(\theta \boldsymbol{a},\Z^{n})<\min\left\{\gamma\|\theta \boldsymbol{a}\|_{2},\alpha\right\}\right\} .
\]
Note that the requirement that the distance is smaller than $\gamma\|\theta \boldsymbol{a}\|_{2}$
forces us to consider only non-trivial integer points as approximations
of $\theta \boldsymbol{a}$. 
\end{definition}
The following proposition, which appears in \cite{rudelson2008littlewood}, shows that vectors with large LCD have small L\'evy concentration function on scales which are larger than $\Omega(1/\text{LCD})$. Here, for completeness, we reproduce a particularly simple proof for the Rademacher case from the lecture notes \cite{rudelson2013lecture}; this is essentially the only case that will be needed in this paper. 
\begin{proposition}[Theorem 6.2 in \cite{rudelson2013lecture}]
\label{prop:LCD-controls-sbp}
Let $\epsilon_{1},\dots,\epsilon_{n}$ denote i.i.d. Rademacher random
variables. Consider a unit vector $\boldsymbol{a}=(a_{1},\dots,a_{n})\in\mathbb{S}^{n-1}$.
Let $S:=\sum_{i=1}^{n}\epsilon_{i}a_{i}$. Then, for every $\alpha>0$,
and for 
\[
\delta\geq\frac{(4/\pi)}{\LCD_{\gamma,\alpha}(\boldsymbol{a})},
\]
we have 
\[
\L(S,\delta)\lesssim\frac{\delta}{\gamma}+\exp(-\alpha^{2}/2).
\]
\end{proposition}
\begin{proof}
We start by using Ess\'een's inequality (\cite{esseen1966kolmogorov}), which estimates the L\'evy
concentration function of a random variable in terms of its characteristic
function as follows: 
\[
\L(X,1)\lesssim\int_{-2}^{2}\left|\E\left[\exp(i\theta X)\right]\right|d\theta
\]
Then, we have
\begin{align*}
\L(S,\delta) & =\L(S/\delta,1)\\
 & \lesssim\int_{-2}^{2}\left|\E\left[\exp(i\theta S/\delta)\right]\right|d\theta\\
 & =\int_{-2}^{2}\prod_{j=1}^{n}\left|\E[\exp(ia_{j}\epsilon_{j}\theta/\delta)\right|d\theta\\
 & =\int_{-2}^{2}\prod_{j=1}^{n}\left|\cos(a_{j}\theta/\delta)\right|d\theta\\
 & \leq\int_{-2}^{2}\prod_{j=1}^{n}\exp\left(-\frac{1}{2}\left(1-\cos^{2}(a_{j}\theta/\delta)\right)\right)d\theta\\
 & =\int_{-2}^{2}\prod_{j=1}^{n}\exp\left(-\frac{1}{2}\sin^{2}(a_{j}\theta/\delta)\right)d\theta\\
 & \leq\int_{-2}^{2}\prod_{j=1}^{n}\exp\left(-\frac{1}{2}\min_{q\in\Z}\left|\frac{2\theta}{\pi\delta}a_{j}-q\right|^{2}\right)d\theta,
\end{align*}
where in the fourth line, we have used the inequality $|x|\leq\exp\left(-\frac{1}{2}(1-x^{2})\right)$,
and in the last line, we have used the pointwise inequality $|\sin(x)|\leq\min_{q\in\Z}\left|\frac{2}{\pi}x-q\right|.$
Thus, we see that 
\begin{align*}
\L(S,\delta) & \lesssim\int_{-2}^{2}\exp\left(-h^{2}(\theta)/2\right)d\theta,
\end{align*}
where 
\[
h(\theta):=\min_{\boldsymbol{p}\in\Z^{n}}\bigg\|\frac{2\theta}{\pi\delta}\boldsymbol{a}-\boldsymbol{p}\bigg\|_{2}.
\]
Since, by assumption, $4/(\pi\delta)\leq\LCD_{\gamma,\alpha}(\boldsymbol{a})$,
it follows that for any $\theta\in[-2,2]$, 
\[
h(\theta)\geq\min\left(\gamma\frac{2\theta}{\pi\delta}\|\boldsymbol{a}\|_{2},\alpha\right)=\min\left(\gamma\frac{2\theta}{\pi\delta},\alpha\right),
\]
so that 
\begin{align*}
\L(S,\delta) & \lesssim\int_{-2}^{2}\left(\exp\left(-(2\gamma\theta/\pi\delta)^{2}/2\right)+\exp(-\alpha^{2}/2)\right)d\theta\\
 & \lesssim\frac{\delta}{\gamma}+\exp(-\alpha^{2}/2),
\end{align*}
as desired. 
\end{proof}

\subsection{Operator norm of random Rademacher matrices and invertibility on a single vector}
We will make use of the following two results, which may be proved in a straightforward manner using standard concentration and epsilon-net arguments. Later, in \cref{prop:boundrestricted-op-norm} and \cref{prop:invertibility-single-rreg}, we will provide proofs of analogous results for the random matrix model under consideration there.\\ 

The first result is a bound on the standard $\ell^{2}\to \ell^{2}$ operator norm of a typical realization of $M_n$. 
\begin{proposition}[See, e.g., Proposition 4.4 in \cite{rudelson2013lecture}]
\label{prop:bound-operator-norm}
There exist absolute constants $C_{\ref{prop:bound-operator-norm}}> 1, c_{\ref{prop:bound-operator-norm}}>0$ for which the following holds. 
For all $t\geq C_{\ref{prop:bound-operator-norm}}$,  
$$\Pr\left(\|M_n\| \geq t\sqrt{n}\right) \lesssim \exp\left(-c_{\ref{prop:bound-operator-norm}}t^{2}n\right).$$
\end{proposition}

The second result shows that, with very high probability, the image of a \emph{fixed} unit vector under $M_n$ does not have norm $o(\sqrt{n})$. 
\begin{lemma}[See, e.g., Corollary 4.6 in  \cite{rudelson2013lecture}]
\label{lemma:invertibility-fixed-vector}
There exists an absolute constant $c_{\ref{lemma:invertibility-fixed-vector}}>0$ for which the following holds. Fix $\boldsymbol{v}\in \S^{n-1}$. Then,
$$\Pr\left(\|M_n \boldsymbol{v}\|_{2} \leq c_{\ref{lemma:invertibility-fixed-vector}}\sqrt{n}\right) \lesssim \exp(-c_{\ref{lemma:invertibility-fixed-vector}}n).$$
\end{lemma}
\subsection{The counting problem in inverse Littlewood-Offord theory}
Our definition of the set $\Gamma$ in \cref{eqn:union-bound-over-ball} and the bound on \cref{eqn:matrix-atom} rely on the approach to the counting problem in inverse Littlewood-Offord theory developed in \cite{FJLS2018}. The starting point of this approach is a classical anti-concentration inequality due to Hal\'asz, which bounds the largest atom probability of an integer vector in terms of its `arithmetic structure'. In order to state this inequality, we need the following definition. Throughout this section, we will work over $\F_{p}$ (the reader should view $p$ as a  `large' (depending on $n$) prime) instead of over $\Z$.  

\begin{definition}
  Suppose that $\boldsymbol{a}\in \F_{p}^{n}$ for $n\in \N$ and an odd prime $p$, and let $k \in \N$. We define $R_k^*(\boldsymbol{a})$ to be the number of solutions to
  \[
    \pm a_{i_1}\pm a_{i_2}\dotsb \pm a_{i_{2k}}= 0 \mod p,
  \]
  where repetitions are allowed in the choice of $i_1,\dots,i_{2k} \in [n]$ and such 
  that $|\{i_1, \dotsc, i_{2k}\}| > (1.01)k$.
\end{definition}
\begin{remark}
\label{rmk:upper-bound-t}
Let $\F_p^*$ denote the set of all finite-dimensional vectors with coefficients in $\F_p$. Then, for every vector $\boldsymbol{a}\in \F^{*}_{p}$  and for every $k\in \N$, we have the trivial bound 
$$R_k^*(\boldsymbol{a}) \leq 2^{2k}\cdot |\boldsymbol{a}|^{2k},$$
where $|\boldsymbol{a}|$ denotes the number of components of $\boldsymbol{a}$. Indeed,
there are at most $|\boldsymbol{a}|^{2k}$ ways of choosing indices $i_1,\dots,i_{2k}\in [|\boldsymbol{a}|]$, and at most $2^{2k}$ ways of choosing a sign pattern which will satisfy the required equation for a given choice of indices. 
\end{remark}

\begin{theorem}[Hal\'asz's inequality over $\F_p$, see Theorem 1.4 in \cite{FJLS2018}] 
  \label{thm:halasz-fp}
  There exists a constant $C_{\ref{thm:halasz-fp}}$ such that the following holds for every odd prime $p$, integer $n$, and vector $\boldsymbol{a}:=(a_1,\dotsc, a_n) \in \F_p^{n}\setminus \{\boldsymbol{0}\}$. Suppose that an integer $0\leq k \leq n/2$ and positive real $M$ satisfy $30M \leq |\supp(\boldsymbol{a})|$ and $80kM \leq n$. Then,
  \[
    \rho_{\F_p}(\boldsymbol{a})\leq \frac{1}{p}+\frac{C_{\ref{thm:halasz-fp}}R_k^{*}(\boldsymbol{a}) + C_{\ref{thm:halasz-fp}}(40k^{0.99}n^{1.01})^{k}}{2^{2k} n^{2k} M^{1/2}} + e^{-M}.
  \]
Here, $\rho_{\F_p}(\boldsymbol{a})$ denotes the largest atom probability of $\boldsymbol{a}$ over $\F_p$.
\end{theorem}
The next theorem bounds the number of vectors over $\F_p^{n}$ which have no `large' subvector with `small' $R_k^{*}$, and is a straightforward consequence of Theorem 1.7 in \cite{FJLS2018}. Later, we will see that this readily translates to a good upper bound on the number of vectors in $\F_{p}^{n}$ with given largest atom probability.  
\begin{theorem}[See also Lemma 3.3 in \cite{FJLS2018}]
\label{theorem:counting}
Let $p$ be an odd prime and let $k \in \N, s_1\geq s_2\in [n], t\in [p]$. 
Let 
$$\Bad^{s_1}_{k,s_2,\geq t}(n):= \left\{\boldsymbol{a} \in \F_{p}^{n} : |\supp(\boldsymbol{a})|\geq s_1, \forall \boldsymbol{b}\subset \boldsymbol{a} \text{ s.t. } |\supp(\boldsymbol{b})|\geq s_{2} \text{ we have } R^*_k(\boldsymbol{b})\geq t\cdot \frac{2^{2k}\cdot |\boldsymbol{b}|^{2k}}{p}\right\}.$$
Then, 
$$|\Bad^{s_1}_{k,s_2,\geq t}(n)| \leq (200)^{n}\left(\frac{s_2}{s_1}\right)^{2k-1}p^{n}t^{-n+s_2}.$$  
\end{theorem}
\begin{proof}
Let us first fix an $S\subseteq [n]$ with $|S|\geq s_1$ and count only those vectors $\boldsymbol{a}$ with $\supp(\boldsymbol{a}) = S$. Define
$$\Bad_{k,s_2,\geq t}(s_1):= \left\{\boldsymbol{a} \in \F_{p}^{s_1} : \forall \boldsymbol{b}\subset \boldsymbol{a} \text{ s.t. } |\supp(\boldsymbol{b})|\geq s_{2} \text{ we have } R^*_k(\boldsymbol{b})\geq t\cdot \frac{2^{2k}\cdot |\boldsymbol{b}|^{2k}}{p}\right\}.$$
Since $\boldsymbol{a}\in \Bad^{s_1}_{k,s_2,\geq t}(n)$, it follows that $\boldsymbol{a}|_{S} \in \Bad_{k,s_2,\geq t}(s_1)$. Hence, Theorem 1.7 in \cite{FJLS2018} shows that the number of choices for $\boldsymbol{a}|_S$ is at most
$$\left(\frac{s_2}{s_1}\right)^{2k-1}\left(0.01 t\right)^{s_2}\left(\frac{100p}{t}\right)^{s_1} \leq 100^{n}\left(\frac{s_2}{s_1}\right)^{2k-1}p^{n}t^{-n+s_2}.$$
Finally, summing over all the at most $2^{n}$ possible choices for $S$ gives the desired conclusion.
\end{proof}
We conclude this subsection by noting that, by \cref{rmk:upper-bound-t}, any vector $\boldsymbol{a}\in \F_p^{n}$ with $|\supp(\boldsymbol{a})| \geq s_1$ must also lie in at least one of the sets $\Bad_{k,s_1,\geq t}^{s_1}(n)$, where $t$ ranges over integers from $0$ to $p$. 

\section{Proof of \cref{thm:lsv-iid}}
\label{sec:proof-thm-iid}
Throughout this section, we will take $\alpha:= n^{1/4}$ and $\gamma := c_{\ref{lemma:invertibility-fixed-vector}}/100C_{\ref{prop:bound-operator-norm}}$. Moreover, since \cref{thm:lsv-iid} is trivially true for $\eta \geq n^{-3/2}$, we will henceforth assume that $2^{-n^{0.0001}} \leq \eta < n^{-3/2}$. We decompose the unit sphere $\S^{n-1}$ into $\Gamma^{1}(\eta) \cup \Gamma^{2}(\eta)$, where
$$\Gamma^{1}(\eta):= \left\{\boldsymbol{a}\in \S^{n-1}: \LCD_{\alpha,\gamma}(\boldsymbol{a}) \geq n^{3/4}\cdot \eta^{-1} \right\} $$
and $\Gamma^{2}(\eta) := \S^{n-1}\setminus \Gamma^{1}(\eta)$. Accordingly, we have
\begin{align}
    \Pr\left(s_n(M_n)\leq \eta\right) \leq \Pr\left(\exists \boldsymbol{a}\in \Gamma^{1}(\eta): \|M_{n}\boldsymbol{a}\|_{2} \leq \eta\right) + \Pr\left(\exists \boldsymbol{a}\in \Gamma^{2}(\eta): \|M_{n}\boldsymbol{a}\|_{2} \leq \eta\right).
\end{align}
Therefore, \cref{thm:lsv-iid} follows from the following two propositions and the union bound. 
\begin{proposition}
\label{prop:eliminate-large-LCD}
$\Pr\left(\exists \boldsymbol{a} \in \Gamma^{1}(\eta): \|M_{n}\boldsymbol{a}\|_{2}\leq \eta \right) \lesssim \eta n^{3/2} + n\exp(-\sqrt{n}/2).$
\end{proposition}
\begin{proposition}
\label{prop:eliminate-small-LCD}
$\Pr\left(\exists \boldsymbol{a} \in \Gamma^{2}(\eta): \|M_{n}\boldsymbol{a}\|_{2}\leq \eta \right) \lesssim \exp(-c_{\ref{prop:eliminate-small-LCD}}n).$
\end{proposition}
The proof of \cref{prop:eliminate-large-LCD} is relatively simple, and follows from a conditioning argument developed in \cite{litvak2005smallest}, once we observe the crucial fact (\cref{prop:LCD-controls-sbp}) that for any $\boldsymbol{a}\in \Gamma^{1}(\eta)$, $\L(\sum_{i=1}^{n}\epsilon_ia_i,\delta) \lesssim \delta + \exp(-\sqrt{n}/2)$ for all $\delta \geq (4/\pi)\eta\cdot n^{-3/4}$. 
\begin{proof}[Proof of \cref{prop:eliminate-large-LCD} (following \cite{litvak2005smallest, tao2009inverse})]
Since $M_n^{T}$ and $M_{n}$ have the same singular values, it follows that a necessary condition for a matrix $M_n$ to satisfy the event in \cref{prop:eliminate-large-LCD} is that there exists a unit vector $\boldsymbol{a'}=(a'_{1},\dots,a'_{n})$
such that $\|\boldsymbol{a'}^{T}M_{n}\|_{2}\leq \eta$. To every matrix $M_n$, associate such a vector $\boldsymbol{a'}$ arbitrarily (if one exists) and denote it by $\boldsymbol{a'}_{M_n}$; this leads to a partition of the space of all $\{\pm 1\}$-valued matrices with least singular value at most $\eta$. Then, by taking a union bound, it suffices to show the following. 
\begin{align}
\label{eqn:intersected-event}
\Pr\left(\exists \boldsymbol{a}\in \Gamma^{1}(\eta): \|M_n \boldsymbol{a}\|_{2} \leq \eta 
\bigwedge \|\boldsymbol{a'}_{M_n}\|_{\infty} = |a'_n| \right) \lesssim \eta \sqrt{n} + \exp(-\sqrt{n}/2).
\end{align}
To this end, we expose the first $n-1$ rows $X_{1},\dots,X_{n-1}$ of $M_{n}$. Note that if there is some $\boldsymbol{a}\in\Gamma^{1}(\eta)$ satisfying $\|M_{n}\boldsymbol{a}\|_{2}\leq \eta$,
then there must exist a vector $\boldsymbol{y}\in \Gamma^{1}(\eta)$, depending only on
the first $n-1$ rows $X_{1},\dots,X_{n-1}$, such that 
\[
\left(\sum_{i=1}^{n-1}(X_{i}\cdot \boldsymbol{y})^{2}\right)^{1/2}\leq \eta.
\]
In other words, once we expose the first $n-1$ rows of the matrix, either the matrix cannot be extended to one satisfying the event in \cref{prop:eliminate-large-LCD}, or there is some unit vector $\boldsymbol{y} \in \Gamma^{1}(\eta)$, which can be chosen after looking only at the first $n-1$ rows, and which satisfies the equation above. For the rest of the proof, we condition on the first $n-1$ rows $X_1,\dots,X_{n-1}$ (and hence, a choice of  $\boldsymbol{y}$).

For any vector $\boldsymbol{w'}\in \S^{n-1}$ with $w'_n \neq 0$, we can write
\[
X_{n}=\frac{1}{w_{n}'}\left(\boldsymbol{u}-\sum_{i=1}^{n-1}w_{i}'X_{i}\right),
\]
where $\boldsymbol{u}:= \boldsymbol{w'}^{T}M_n$.
Thus, for the event $\{s_n(M_n) \leq \eta\}\bigwedge \{\|\boldsymbol{a'}_{M_n}\|_{\infty} = |a'_n|\}$ to occur, we must necessarily have
\begin{align*}
\left|X_{n}\cdot \boldsymbol{y}\right| & =\inf_{\boldsymbol{w'}\in \S^{n-1}, w'_n \neq 0}\frac{1}{|w_{n}'|}\left|\boldsymbol{u}\cdot \boldsymbol{y}-\sum_{i=1}^{n-1}w_{i}'X_{i}\cdot \boldsymbol{y}\right|\\
 &\leq  
 \frac{1}{|a_{n}'|}\left(\|\boldsymbol{a'}_{M_n}^{T}M_{n}\|_{2}\|\boldsymbol{y}\|_{2}+\|\boldsymbol{a'}_{M_n}\|_{2}\left(\sum_{i=1}^{n-1}(X_{i}\cdot \boldsymbol{y})^{2}\right)^{1/2}\right)\\
 &\leq \eta \sqrt{n}\left(\|\boldsymbol{y}\|_{2} + \|\boldsymbol{a'}_{M_n}\|_{2}\right) \leq 2\eta \sqrt{n},
\end{align*}
where the second line is due to the Cauchy-Schwarz inequality and the particular choice $\boldsymbol{w'}=\boldsymbol{a'}_{M_n}$.
It follows, by definition, that the probability in \cref{eqn:intersected-event} is bounded by $\L(\boldsymbol{y},2\eta \sqrt{n})$, and hence, by
$$\L(\boldsymbol{y},2\eta \sqrt{n}) \lesssim \eta \sqrt{n} + \exp(-\sqrt{n}/2), $$
which completes the proof. 
\end{proof}
The proof of \cref{prop:eliminate-small-LCD} is the content of the next three subsections.
\subsection{Reduction to integer vectors}
 Here, we present the initial crucial step, which consists of efficiently passing from vectors on the unit sphere to integer vectors. 
\begin{proposition}
\label{lemma:reduction-to-integer}
With notation as above, we have
\begin{align*}
\Pr\left(\exists \boldsymbol{a} \in \Gamma^{2}(\eta):\|M_{n}\boldsymbol{a}\|_{2}\leq \eta \right) 
&\lesssim e^{-c_{\ref{prop:bound-operator-norm}}n} + \\ \Pr(\exists \boldsymbol{w} \in (\Z^{n}\setminus\{\boldsymbol{0}\}) \cap [-2\eta^{-1}n^{3/4}, 2\eta^{-1}n^{3/4}]^{n} &: \|M_{n}\boldsymbol{w}\|_{2}\leq \min\{4\gamma  C_{\ref{prop:bound-operator-norm}}\sqrt{n}\|\boldsymbol{w}\|_{2},2C_{\ref{prop:bound-operator-norm}} \alpha\sqrt{n}\}).
\end{align*}
\end{proposition}
\begin{proof}
Since by \cref{prop:bound-operator-norm}, $\Pr\left(\|M_{n}\| \geq C_{\ref{prop:bound-operator-norm}} \sqrt{n}\right) \lesssim \exp(-c_{\ref{prop:bound-operator-norm}} C_{\ref{prop:bound-operator-norm}}^{2}n)$, we may henceforth restrict to the complement of this event. 
Let $\boldsymbol{a}\in \Gamma^{2}(\eta)$. Then, by definition, there exists some $0<\theta \leq \LCD_{\alpha,\gamma}(\boldsymbol{a}) \leq n^{3/4}\eta^{-1}$ and some $\boldsymbol{w}\in \Z^{n}\setminus\{\boldsymbol{0}\}$ such that $\|\theta \boldsymbol{a} - \boldsymbol{w}\|_{2} \leq \min\{\gamma \theta,\alpha\}$. Thus, if $\|M_{n}\boldsymbol{a}\|_{2}\leq \eta$, it follows from the triangle inequality that 
\begin{align*}
\|M_{n}\boldsymbol{w}\|_{2} & =\|M_{n}(\boldsymbol{w}-\theta \boldsymbol{a})+M_{n}(\theta \boldsymbol{a})\|_{2}\\
 & \leq\|M_{n}\|\cdot\|\theta \boldsymbol{a}-\boldsymbol{w}\|_{2}+\theta\cdot \|M_{n}\boldsymbol{a}\|_{2}\\
 & \leq C_{\ref{prop:bound-operator-norm}}\sqrt{n}\cdot\min\{\gamma \theta, \alpha\}+\theta\eta\\
 &\leq 2C_{\ref{prop:bound-operator-norm}}\sqrt{n}\cdot \min\{\gamma \theta, \alpha\},
\end{align*}
where the last inequality follows since $\eta \leq \gamma \sqrt{n}$ and $\theta \eta \leq n^{3/4} \leq \sqrt{n}\alpha$.
The desired conclusion now follows from the straightforward case analysis below.\\ \\
\textbf{Case I: $\gamma\theta \leq \alpha$}. In this case, $\boldsymbol{w}$ is a non-zero integer vector of norm $\|\boldsymbol{w}\|_{2} = \theta(1\pm \gamma)$ satisfying
$$\|M_{n}\boldsymbol{w}\|_{2} \leq 2\gamma C_{\ref{prop:bound-operator-norm}} \sqrt{n}\theta \leq \min\{4\gamma C_{\ref{prop:bound-operator-norm}}\sqrt{n}\|\boldsymbol{w}\|_{2}, 2C_{\ref{prop:bound-operator-norm}}\alpha\sqrt{n}\},$$
where the last inequality uses $\theta \leq \|\boldsymbol{w}\|_{2}$ and $\gamma \theta \leq \alpha$.\\ \\
\textbf{Case II: $\gamma\theta > \alpha$}. In this case, $\boldsymbol{w}$ is a non-zero integer vector of norm $\|\boldsymbol{w}\|_{2} = \theta(1\pm \gamma) \geq \gamma^{-1}\alpha/2 $ satisfying 
$$\|M_{n}\boldsymbol{w}\|_{2}\leq 2C_{\ref{prop:bound-operator-norm}}\alpha \sqrt{n} \leq \min\{2C_{\ref{prop:bound-operator-norm}}\gamma^{-1}\alpha\gamma \sqrt{n}, 2C_{\ref{prop:bound-operator-norm}}\alpha\sqrt{n}\} \leq \min\{4\gamma C_{\ref{prop:bound-operator-norm}}\sqrt{n}\|\boldsymbol{w}\|_{2},2C_{\ref{prop:bound-operator-norm}}\alpha\sqrt{n}\}.$$
\end{proof}

\subsection{Dealing with sparse integer vectors}
The goal of this subsection is to prove the following lemma, which follows from \cref{lemma:invertibility-fixed-vector} and a simple union bound. Throughout this subsection and the next one, $p = 2^{n^{0.001}}$ is a prime. Note, in particular, that $p \gg \eta^{-1}n^{3/4}$. 
\begin{lemma}
\label{lemma:sparse-vectors}
$\Pr\left(\exists \boldsymbol{w} \in (\Z^{n}\setminus\{\boldsymbol{0}\})\cap [-p,p]^{n}, |\supp(\boldsymbol{w})|\leq n^{0.99}: \|M_{n}\boldsymbol{w}\|_{2} \leq 4\gamma C_{\ref{prop:bound-operator-norm}} \sqrt{n}\|\boldsymbol{w}\|_{2}\right) \lesssim \exp(-c_{\ref{lemma:invertibility-fixed-vector}}n/2).$
\end{lemma}
\begin{proof}
The number of vectors $\boldsymbol{w} \in (\Z^{n}\setminus\{\boldsymbol{0}\})\cap [-p,p]^{n}$ with support of size no more than $n^{0.99}$ is at most
$$\binom{n}{n^{0.99}}(3p)^{n^{0.99}} \ll 2^{n^{0.992}}.$$
By \cref{lemma:invertibility-fixed-vector}, for any such vector, 
$$\Pr\left(\|M_{n}\boldsymbol{w}\|_{2} \leq 4\gamma C_{\ref{prop:bound-operator-norm}}\sqrt{n}\|\boldsymbol{w}\|_{2}  \right) \leq \Pr\left(\|M_{n}\boldsymbol{w}\|_{2} \leq c_{\ref{lemma:invertibility-fixed-vector}}\sqrt{n}\|\boldsymbol{w}\|_{2} \right) \lesssim \exp(-c_{\ref{lemma:invertibility-fixed-vector}}n).$$
Therefore, the union bound gives the desired conclusion. 
\end{proof}
\subsection{Dealing with non-sparse integer vectors}
Throughout this subsection, we fix $k = n^{0.01}$, $s_1 = s_2 = n^{0.99}$. 
It remains to deal with integer vectors with support of size at least $n^{0.99}$. Formally, let
$$\boldsymbol{W}:= \left\{\boldsymbol{w}\in (\Z^{n}\setminus \{\boldsymbol{0}\})\cap [-\eta^{-4},\eta^{-4}]^{n}: |\supp(\boldsymbol{w})| \geq n^{0.99} \right\}.$$

In view of \cref{lemma:reduction-to-integer} and \cref{lemma:sparse-vectors}, and since $\eta \leq n^{-3/2}$, the following proposition suffices to prove \cref{prop:eliminate-small-LCD}.
\begin{proposition}
\label{prop:eliminate-smallLCD-integer}
$\Pr\left(\exists \boldsymbol{w} \in \boldsymbol{W} : \|M_{n}\boldsymbol{w}\|_{2} \leq 2C_{\ref{prop:bound-operator-norm}}n^{3/4} \right) \lesssim n^{-0.01n}.$
\end{proposition} 
This will be accomplished by a union bound, following the strategy outlined in \cref{eqn:union-bound-over-ball}. Note that for our choice of parameters, the natural map
$$\iota: \boldsymbol{W} \to \F_{p}^{n}$$
is injective, and we will often abuse notation by using $\boldsymbol{w}$ to denote $\iota(\boldsymbol{w})$. This identification enables us to make the following definition. 
\begin{definition}
For an integer $t \in [p]$, let
$$\boldsymbol{W}_{t}:= \left\{\boldsymbol{w} \in \boldsymbol{W}: \iota(\boldsymbol{w}) \in \Bad_{k,s_2,\geq t-1}^{s_1}(n)\setminus \Bad_{k,s_2, \geq t}^{s_1}(n) \right\}.$$
\end{definition}
We will need the following two lemmas. 
\begin{lemma}
\label{lemma:usable-Halasz}
There exists an absolute constant $C_{\ref{lemma:usable-Halasz}} > 1$ such that, for our choice of parameters, if $\boldsymbol{w}\in \boldsymbol{W}_{t}$, then
$$\rho(\boldsymbol{w}) \leq \frac{C_{\ref{lemma:usable-Halasz}}}{p}\left(\frac{t}{n^{0.48}} + 1\right).$$
\end{lemma}
\begin{proof}
Since $\rho(\boldsymbol{w}) \leq \rho_{\F_p}(\iota(\boldsymbol{w}))=:\rho_{\F_p}(\boldsymbol{w})$, it suffices to prove the statement for the latter quantity. This, in turn, follows from a direct application of Hal\'asz's inequality (\cref{thm:halasz-fp}). Indeed, since $\boldsymbol{w} \notin \Bad^{s_1}_{k,s_2,\geq t}(n)$, there exists some $\boldsymbol{b}\subset \boldsymbol{a}$ such that $|\supp(\boldsymbol{b})| \geq s_2$ and 
$$R_k^*(\boldsymbol{b}) \leq t\cdot \frac{2^{2k}\cdot |\boldsymbol{b}|^{2k}}{p}.$$
Moreover, for our choice of parameters, we have
$$(40k^{0.99}n^{1.01})^{k} \ll \frac{2^{2k}s_{2}^{2k}}{\sqrt{p}} \leq t\cdot \frac{2^{2k}\cdot |\boldsymbol{b}|^{2k}}{p}.$$
Hence, applying Halasz's inequality to the $|\boldsymbol{b}|$-dimensional vector $\boldsymbol{b}$ with $M = n^{0.96}$ (note that this choice of $M$ satisfies the conditions $30M \leq s_2 \leq |\supp(\boldsymbol{b})|$ and $80kM \leq s_2 \leq |\boldsymbol{b}|$ needed to apply Hal\'asz's inequality), and observing that (trivially) $\rho_{\F_p}(\boldsymbol{a}) \leq \rho_{\F_p}(\boldsymbol{b})$, we get 
\begin{align*}
    \rho_{\F_p}(\boldsymbol{a}) 
    &\lesssim \frac{1}{p} + \frac{t\cdot \frac{2^{2k}\cdot |\boldsymbol{b}|^{2k}}{p}}{2^{2k}|\boldsymbol{b}|^{2k}n^{0.48}} + e^{-n^{0.96}}\\
    &\lesssim\frac{1}{p}\left( \frac{t}{n^{0.48}} + 1\right), 
\end{align*}
as desired. 
\end{proof}
\begin{lemma}
\label{corollary:counting}
For our choice of parameters, 
$$|\boldsymbol{W}_{t}| \leq (300)^{n}\left(\frac{p}{t}\right)^{n}.$$
\end{lemma}
\begin{proof}
By definition, any $\boldsymbol{w}\in \boldsymbol{W}_{t}$ satisfies $\iota(\boldsymbol{w}) \in \Bad^{s_1}_{k,s_2, \geq t-1}(n)$.  
Hence, by \cref{theorem:counting}, the number of possible such vectors $\iota(\boldsymbol{w})$ is at most 
$$(200)^{n}\left(\frac{p}{t-1}\right)^{n}p^{s_2} \leq (300)^{n}\left(\frac{p}{t}\right)^{n}.$$
Using the injectivity of $\iota$ gives the desired conclusion.
\end{proof}
Finally, we are in a position to prove \cref{prop:eliminate-smallLCD-integer}. As discussed at the start of this subsection, this completes the proof of \cref{prop:eliminate-small-LCD} and hence, the proof of \cref{thm:lsv-iid}.
\begin{proof}[Proof of \cref{prop:eliminate-smallLCD-integer}]
We begin by noting that every $\boldsymbol{w} \in \boldsymbol{W}$ 
has $\rho(\boldsymbol{w}) \geq \eta^{4}n^{-1}/3$. Indeed, for any such vector $\boldsymbol{w}$, $\sum_{i=1}^{n}\epsilon_i w_i$ can take on at most $3n\eta^{-4}$ values, so that the claim follows from the pigeonhole principle. Since $\eta^{4}n^{-1}/3 \gg 1/\sqrt{p}$, it follows from \cref{lemma:usable-Halasz} that $\boldsymbol{W}_{t} = \emptyset$ for all $t \leq \sqrt{p}$. 

On the other hand, using \cref{eqn:union-bound-over-ball} with 
$\Gamma = \boldsymbol{W}_{t}$ and $C(n) = 2C_{\ref{prop:bound-operator-norm}}n^{1/4}$, it follows from \cref{lemma:usable-Halasz} and \cref{corollary:counting} that for all $t\geq \sqrt{p}$, the probability that the image of any vector in $\boldsymbol{W}_{t}$ under $M_n$ lies in the ball of radius $2C_{\ref{prop:bound-operator-norm}}n^{3/4}$ centered around the origin is at most
$$(200C_{\ref{prop:bound-operator-norm}}n^{1/4})^{n}|\boldsymbol{W}_{t}|\left(\frac{2C_{\ref{lemma:usable-Halasz}}t}{pn^{0.48}}\right)^{n}\leq (200C_{\ref{prop:bound-operator-norm}}n^{1/4})^{n}(300)^{n}\left(\frac{p}{t}\right)^{n}\left(\frac{2C_{\ref{lemma:usable-Halasz}}t}{pn^{0.48}}\right)^{n} \ll n^{-0.01n}.$$
Finally, taking the union bound over integers $t\in [\sqrt{p},p]$ completes the proof.
\end{proof}

\section{Proof of \cref{thm:lsv-rreg}}
\label{sec:proof-thm-rreg}
\subsection{Outline of the proof}
A major difference between the proofs of \cref{thm:lsv-iid} and \cref{thm:lsv-rreg} is that \cref{prop:bound-operator-norm} and \cref{lemma:invertibility-fixed-vector} are no longer available to us; indeed, the operator norm of $Q_n$ is $n/2$, whereas the standard proof of \cref{lemma:invertibility-fixed-vector} does not immediately go through since the random variables $\langle Q_n \boldsymbol{v} ,e_i \rangle$ might not have their largest atom probability bounded away from $1$ (for instance, this is the case when $\boldsymbol{v}$ is the all ones vector). A large part of the proof is devoted to circumventing these issues. 

To overcome the first problem, we exploit the presence of a `spectral gap'. Namely, we show (\cref{prop:boundrestricted-op-norm}) that, while the operator norm of $Q_n$ is $n/2$, the operator norm of $Q_n$ restricted to the hyperplane $\boldsymbol{H}:= \{\boldsymbol{v}\in\R^{n}:\sum_{i=1}^{n}v_{i}=0\}$ is at most $n^{0.51}$ with high probability. The utility of this is that one can slightly modify the best integer approximation to a vector (guaranteed by the definition of the LCD) in such a way that the difference/approximation error is contained almost entirely in $\boldsymbol{H}$ (\cref{prop:reduction-to-integer-rreg}); since the only place where we need the operator norm is to bound the norm of $Q_n$ applied to this difference, it follows that the `effective operator norm' for our purpose is at most $n^{0.51}$.       

To overcome the second obstacle, we prove (\cref{prop:invertibility-single-rreg}) a concentration inequality for sums of low-degree polynomials on slices of the Boolean hypercube. Our proof combines the classical hypercontractive estimates for polynomials on the Boolean hypercube with more recent hypercontractive estimates for polynomials on slices of the Boolean hypercube, and may be of independent interest.  

Even given these additional tools, the remainder of the proof is not as straightforward as the proof of \cref{thm:lsv-iid}; after our reduction to integer vectors (\cref{prop:reduction-to-integer-rreg}), we will need to exploit the approach in \cite{FJLS2018} (used there to study the singularity probability of $Q_n$) in order to get to the setting of \cref{eqn:union-bound-over-ball,eqn:matrix-atom} and complete the proof. 

\subsection{Bounding the operator norm restricted to $\boldsymbol{H}$}
\begin{lemma}
\label{prop:boundrestricted-op-norm}
There exist absolute constants $C_{\ref{prop:boundrestricted-op-norm}}>1$ and $c_{\ref{prop:boundrestricted-op-norm}}>0$ for which the following
holds. For all $t\geq C_{\ref{prop:boundrestricted-op-norm}}$, 
\[
\Pr\left(\sup_{\boldsymbol{v}\in \boldsymbol{H}\cap \S^{n-1}}\|Q_{n}\boldsymbol{v}\|_{2}\geq tn^{0.51}\right)\lesssim\exp\left(-c_{\ref{prop:boundrestricted-op-norm}}t^{2}n^{1.02}\right).
\]
\end{lemma}
\begin{proof}
Let $M_{n}$ denote a uniformly random $n\times n$ $\{\pm1\}$-valued
matrix. We will use the easy observation that $Q_{n}\sim\left(2^{-1}(\boldsymbol{1}_{n\times n}+M_{n})\right)|\{M_{n}\boldsymbol{1}=\boldsymbol{0}\}$,
where $\boldsymbol{1}_{n\times n}$ denotes the $n\times n$ all ones
matrix and $\boldsymbol{1}$ denotes the all ones vector. Since $\Pr\left(M_{n}\boldsymbol{1}=\boldsymbol{0}\right)\geq\left(\frac{1}{\sqrt{100n}}\right)^{n} = \exp(-\Theta(n\log{n}))$,
it suffices to show that 
\[
\Pr\left(\sup_{\boldsymbol{v}\in \boldsymbol{H}\cap \S^{n-1}}\|2^{-1}(\boldsymbol{1}_{n\times n}+M_{n})\boldsymbol{v}\|_{2}\geq tn^{0.51}\right)=\Pr\left(\sup_{\boldsymbol{v}\in \boldsymbol{H}\cap \S^{n-1}}\|M_{n}\boldsymbol{v}\|_{2}\geq 2tn^{0.51}\right)\lesssim \exp\left(-\Omega(t^{2}n^{1.02})\right),
\]
where the first equality uses that $\boldsymbol{1}_{n\times n}\boldsymbol{v} = \boldsymbol{0}$ for any $\boldsymbol{v}\in\boldsymbol{H}$.
But from \cref{prop:bound-operator-norm}, we have for all $t \geq C_{\ref{prop:bound-operator-norm}}$ that
$$\Pr\left(\sup_{\boldsymbol{v}\in \boldsymbol{H}\cap \S^{n-1}}\|M_{n}\boldsymbol{v}\|_{2}\geq 2tn^{0.51}\right)\leq \Pr\left(\|M_{n}\|_{2} \geq 2tn^{0.01}\sqrt{n}\right) \lesssim \exp\left(-4c_{\ref{prop:bound-operator-norm}}t^{2}n^{1.02}\right),$$
which completes the proof. 
\end{proof}
\subsection{Invertibility on a fixed vector}
\begin{proposition}
\label{prop:invertibility-single-rreg}
For any $\epsilon > 0$, there exists a constant $C_{\ref{prop:invertibility-single-rreg}}:= C_{\ref{prop:invertibility-single-rreg}}(\epsilon) > 1$ for which the following holds.
Fix $\boldsymbol{v}\in\S^{n-1}$. Then, 
\[
\Pr\left(\|Q_{n}\boldsymbol{v}\|_{2}\leq \frac{\sqrt{n}}{2}\|\boldsymbol{v}\|_{2}\right)\leq 
C_{\ref{prop:invertibility-single-rreg}}(\epsilon)\exp\left(-\frac{n^{1-\epsilon}}{4}\right).\]
\end{proposition}
The proof of this proposition will require a few intermediate steps. We begin by computing the expectation of the random variable $\|Q_{n}\boldsymbol{v}\|_{2}^{2}$ for fixed $\boldsymbol{v}=(v_{1},\dots,v_{n})\in\S^{n-1}$. Consider
the random variable $X:=v_{1}(1+x_{1})+\dots+v_{n}(1+x_{n})$, where
$x_{1},\dots,x_{n}$ are $\{\pm1\}$-valued random variables sampled
uniformly from the hyperplane $x_{1}+\dots+x_{n}=0$. Then, for all
$i\in[n]$, the random variables $\langle Q_{n}\boldsymbol{v},e_{i}\rangle$
are independent copies of $X/2$, so that 
$$4\|Q_{n}\boldsymbol{v}\|_{2}^{2} \sim X_{1}^{2}+\dots+X_{n}^{2},$$
where $X_1,\dots,X_n$ are i.i.d. copies of $X$. Since 
\begin{align*}
\E[X^{2}] & =\E\left[\left(\sum_{i=1}^{n}v_{i}+\sum_{i=1}^{n}v_{i}x_{i}\right)^{2}\right]\\
 & =\left(\sum_{i=1}^{n}v_{i}\right)^{2}+\E\left[\left(\sum_{i=1}^{n}v_{i}x_{i}\right)^{2}\right]+2\left(\sum_{i=1}^{n}v_{i}\right)\left(\sum_{i=1}^{n}v_{i}\E[x_{i}]\right)\\
 & =\left(\sum_{i=1}^{n}v_{i}\right)^{2}+\E\left[\left(\sum_{i=1}^{n}v_{i}x_{i}\right)^{2}\right]\\
 & =\left(\sum_{i=1}^{n}v_{i}\right)^{2}+\sum_{i=1}^{n}v_{i}^{2}\E[x_{i}^{2}]+\sum_{i\neq j}v_{i}v_{j}\E[x_{i}x_{j}]\\
 & =\left(\sum_{i=1}^{n}v_{i}\right)^{2}+\sum_{i=1}^{n}v_{i}^{2}-\frac{1}{n-1}\sum_{i\neq j}v_{i}v_{j}\\
 & =\left(\sum_{i=1}^{n}v_{i}\right)^{2}+\left(1+\frac{1}{n-1}\right)\sum_{i=1}^{n}v_{i}^{2}-\frac{1}{n-1}\left(\sum_{i=1}^{n}v_{i}^{2}+\sum_{i\neq j}v_{i}v_{j}\right)\\
 & =\left(\sum_{i=1}^{n}v_{i}\right)^{2}+\frac{n}{n-1}\sum_{i=1}^{n}v_{i}^{2}-\frac{1}{n-1}\left(\sum_{i=1}^{n}v_{i}\right)^{2}\\
 & =\frac{n-2}{n-1}\left(\sum_{i=1}^{n}v_{i}\right)^{2}+\frac{n}{n-1}\sum_{i=1}^{n}v_{i}^{2},
\end{align*}
it follows that
$$\E\left[\|Q_{n}\boldsymbol{v}\|_{2}^{2}\right] = \frac{n^{2}-2n}{4(n-1)}\left(\sum_{i=1}^{n}v_i\right)^{2} + \frac{n^2}{4(n-1)}\sum_{i=1}^{n}v_{i}^{2}.$$
The remainder of the proof consists of showing that the random variable $\|Q_{n}\boldsymbol{v}\|_{2}^{2}$ is sufficiently well-concentrated around its expectation using the standard exponential moment method (Bernstein's trick). 
For this, we need good control on the moments of $X^{2}$. The control for `low' moments is provided by the following hypercontractivity inequality on slices of the Boolean hypercube, which is applicable in our setting since $X$ is a linear polynomial on the central slice of the Boolean hypercube. 
\begin{lemma}[see, e.g., Proposition 2.5 and Corollary 2.6 in \cite{kwan2018anticoncentration}]
For any integer $q\geq 1$,
$$\E[X^{2q}] \leq O_q(1)\left(\E[X^{2}]\right)^{q}.$$
\end{lemma}
For `high' moments, the above estimate is possibly wasteful since the factor $O_q(1)$ could grow too quickly as a function of $q$. However, we can do better by combining the classical hypercontractive estimate for polynomials on the Boolean hypercube with a simple conditioning argument. 
\begin{lemma}
For any integer $q\geq 1$,
$$\E[X^{2q}] \leq 100\sqrt{n}\cdot(4q)^{q} \left(\E[X^2]\right)^{q}.$$
\end{lemma}
\begin{proof}
Consider the random variable
$Y:=v_{1}(1+\epsilon_{1})+\dots+v_{n}(1+\epsilon_{n})$, where $\epsilon_{1},\dots,\epsilon_{n}$
are i.i.d. Rademacher random variables, and observe as before that
$X\sim Y|\{\epsilon_{1}+\dots+\epsilon_{n}=0\}$. Since $Y$ is a linear form on the Boolean hypercube $\{\pm1\}^{n}$
equipped with the uniform measure, it follows from the usual hypercontractive
inequality (see Theorem 9.21 of \cite{o2014analysis}) that for all integers $q\geq1$,
\begin{align*}
\E\left[Y^{2q}\right] & \leq(2q)^{q}\cdot\left(\E\left[Y^{2}\right]\right)^{q}.
\end{align*}
Moreover, a short calculation similar to (but easier than) the one for $X^2$ shows that
$$\E[Y^2] = \left(\sum_{i=1}^{n}v_i\right)^{2} + \sum_{i=1}^{n}v_{i}^{2}.$$
Therefore, we have
\begin{align*}
\E\left[X^{2q}\right] 
 & =\E\left[Y^{2q}|\epsilon_{1}+\dots+\epsilon_{n}=0\right]\\
 & \leq\frac{\E\left[Y^{2q}\right]}{\Pr\left(\epsilon_{1}+\dots+\epsilon_{n}=0\right)}\\
 & \leq 100\sqrt{n}\cdot \E\left[Y^{2q}\right]\\
 & \leq 100\sqrt{n} \cdot (2q)^{q}\cdot \left(\E[Y^{2}]\right)^{q}\\
 & \leq 100\sqrt{n}\cdot (2q)^{q}\cdot (2\E[X^{2}])^{q},
\end{align*}
which gives the desired conclusion.
\end{proof}
Combining these two lemmas immediately gives the following. 
\begin{lemma}
\label{lemma:moment-bound}
For any integer $q\geq 1$,
$$\|X^{2}-\E[X^{2}]\|_{q} \leq \min\left\{{O_{q}(1)},(100\sqrt{n})^{1/q}\cdot 5q\right\}\cdot \E[X^{2}].$$
\end{lemma}
\begin{proof}
By the triangle inequality for the $L^q$-norm, we get that
\begin{align*}
\|X^{2}-\E[X^{2}]\|_{q} 
&\leq \|X^{2}\|_{q} + \|\E[X^{2}]\|_{q}\\
&\leq \min\left\{{O_{q}(1)},(100\sqrt{n})^{1/q}\cdot 4q\right\}\cdot \E[X^{2}] + \E[X^{2}]\\
&\leq \min\left\{{O_{q}(1)},(100\sqrt{n})^{1/q}\cdot 5q\right\}\cdot \E[X^{2}],
\end{align*}
where the second inequality follows from the previous two lemmas. 
\end{proof}
The previous bound on moments can now be used to obtain a useful bound on the moment generating function. 
\begin{lemma}
\label{lemma:MGF-bound}
Let $Z:=\E[X^2] - X^{2}$. Then, for any integer $t\geq 3$ and for any $0<\lambda<1/(40\E[X^{2}])$,
$$\E\left[\exp\left(\lambda Z\right)\right] \leq  1+{O_{t}(1)}\lambda^{2}\E[X^{2}]^{2}+200\sqrt{n}\cdot20^{t}\lambda^{t}\E[X^{2}]^{t}.$$
\end{lemma}
\begin{proof}
For the range of parameters in the statement of the lemma, we have 
\begin{align*}
\E\left[\exp\left(\lambda Z\right)\right] & =\E\left[\sum_{q=0}^{\infty}\frac{\lambda^{q}Z^{q}}{q!}\right]=1+\sum_{q=2}^{\infty}\frac{\lambda^{q}\E[Z^{q}]}{q!}\\
 & \leq1+\sum_{q=2}^{t-1}\frac{\lambda^{q}\|Z\|_{q}^{q}}{q!}+\sum_{q=t}^{\infty}\frac{\lambda^{q}\|Z\|_{q}^{q}}{q!}\\
 & \leq1+{O_{t}(1)}\sum_{q=2}^{t-1}\frac{\lambda^{q}}{q!}\E[X^{2}]^{q}+100\sqrt{n}\sum_{q=t}^{\infty}\frac{\lambda^{q}\cdot(5q)^{q}\E[X^{2}]^{q}}{q!}\\
 & \leq1+{O_{t}(1)}\lambda^{2}\E[X^{2}]^{2}+100\sqrt{n}\sum_{q=t}^{\infty}\left(20\lambda\E[X^{2}]\right)^{q}\\
 & \leq1+{O_{t}(1)}\lambda^{2}\E[X^{2}]^{2}+200\sqrt{n}\cdot20^{t}\lambda^{t}\E[X^{2}]^{t},
\end{align*}
where the third line follows by \cref{lemma:moment-bound}.
\end{proof}
Finally, we are in a position to prove \cref{prop:invertibility-single-rreg}.
\begin{proof}[Proof of \cref{prop:invertibility-single-rreg}]
As above, let $Z:=\E[X^{2}]-X^2$, and let $Z_1,\dots,Z_n$ be i.i.d. copies of $Z$. For any integer $t \geq 3$ and for any $0 < \lambda < 1/(40\E[X^2])$, we have 
\begin{align*}
\Pr\left(\|Q_{n}\boldsymbol{v}\|_{2}\leq\frac{\sqrt{n}}{2}\|\boldsymbol{v}\|_{2}\right) & \leq\Pr\left(\|Q_{n}\boldsymbol{v}\|_{2}^{2}\leq\frac{n}{4}\|\boldsymbol{v}\|_{2}^{2}\right)\leq\Pr\left(4\sum_{i=1}^{n}X_{i}^{2}\leq\frac{n}{4}\|\boldsymbol{v}\|_{2}^{2}\right)\\
 & \leq\Pr\left(\sum_{i=1}^{n}X_{i}^{2}\leq\frac{n}{16}\E[X^{2}]\right)\leq\Pr\left(\sum_{i=1}^{n}Z_{i}\geq15n\E[X^{2}]/16\right)\\
 & \leq\Pr\left(\exp\left(\lambda\sum_{i=1}^{n}Z_{i}\right)\geq\exp\left(15\lambda n\E[X^{2}]/16\right)\right)\\
 & \leq\exp\left(-\frac{15\lambda n\E[X^{2}]}{16}\right)\prod_{i=1}^{n}\E\left[\exp(\lambda Z_{i})\right]\\
 & \leq\exp\left(-\frac{\lambda n\E[X^{2}]}{2}\right)\left(1+{O_{t}(1)}\lambda^{2}\E[X^{2}]^{2}+200\sqrt{n}\cdot20^{t}\lambda^{t}\E[X^{2}]^{t}\right)^{n},
\end{align*}
where the last line follows from \cref{lemma:MGF-bound}. Let $\epsilon > 0$ be fixed as in the statement of the theorem, and take $t\geq 3$ to be the smallest integer for which
$$\sqrt{n}n^{-t\epsilon} \leq n^{-2\epsilon}.$$
Then, for $\lambda = 1/(n^{\epsilon}\E[X^2])$ (which satisfies our assumption on $\lambda$ for all $n$ sufficiently large), we see that
the right hand side is at most 
\begin{align*}
\exp\left(-\frac{n^{1-\epsilon}}{2}\right)\left(1+{O_{t}(1)}n^{-2\epsilon}\right)^{n} & \leq\exp\left(-\frac{n^{1-\epsilon}}{2}+{O_{t}(1)}n^{1-2\epsilon}\right)\\
 & \lesssim\exp\left(-\frac{n^{1-\epsilon}}{4}\right),
\end{align*}
which completes the proof.  \end{proof}
\subsection{Reduction to integer vectors}
Throughout this section, we will take $\alpha:= n^{1/4}$ and 
$\gamma = n^{-1}$. Moreover, since \cref{thm:lsv-rreg} is trivially true for $\eta \geq n^{-2}$, we will henceforth assume that $2^{-n^{0.0001}} \leq \eta < n^{-2}$. We decompose the unit sphere $\S^{n-1}$ into $\Gamma^{1}(\eta) \cup \Gamma^{2}(\eta)$, where
$$\Gamma^{1}(\eta):= \left\{\boldsymbol{a}\in \S^{n-1}: \LCD_{\alpha,\gamma}(\boldsymbol{a}) \geq n^{3/4}\cdot \eta^{-1} \right\} $$
and $\Gamma^{2}(\eta) := \S^{n-1}\setminus \Gamma^{1}(\eta)$. Accordingly, we have
\begin{align}
    \Pr\left(s_n(Q_n)\leq \eta\right) \leq \Pr\left(\exists \boldsymbol{a}\in \Gamma^{1}(\eta): \|Q_{n}\boldsymbol{a}\|_{2} \leq \eta\right) + \Pr\left(\exists \boldsymbol{a}\in \Gamma^{2}(\eta): \|Q_{n}\boldsymbol{a}\|_{2} \leq \eta\right).
\end{align}
Therefore, \cref{thm:lsv-rreg} follows from the following two propositions and the union bound. 
\begin{proposition}
\label{prop:eliminate-large-LCD-rreg}
$\Pr\left(\exists \boldsymbol{a} \in \Gamma^{1}(\eta): \|Q_{n}\boldsymbol{a}\|_{2}\leq \eta \right) \lesssim \eta n^{2} + n^{3/2}\exp(-\sqrt{n}/2).$
\end{proposition}
\begin{proposition}
\label{prop:eliminate-small-LCD-rreg}
$\Pr\left(\exists \boldsymbol{a} \in \Gamma^{2}(\eta): \|Q_{n}\boldsymbol{a}\|_{2}\leq \eta \right) \lesssim \exp(-c_{\ref{prop:eliminate-small-LCD-rreg}}\sqrt{n}).$
\end{proposition}
The proof of \cref{prop:eliminate-large-LCD-rreg} is almost exactly the same as that of \cref{prop:eliminate-large-LCD}. The only difference is that, at the very end, instead of using \cref{prop:LCD-controls-sbp}, we use the following variant.
\begin{proposition}
Let $n\geq 2$ be an even integer. Fix a unit vector $\boldsymbol{a}=(a_{1},\dots,a_{n})\in\mathbb{S}^{n-1}$ and consider the random variable $S:=\sum_{i=1}^{n}y_{i}a_{i}$, where $y_{i}$ are $\{0, 1\}$-valued random variables sampled uniformly from the hyperplane $y_1+\dots + y_n = n/2$. Then, for every $\alpha>0$,
and for 
\[
\delta\geq\frac{(4/\pi)}{\LCD_{\gamma,\alpha}(\boldsymbol{a})},
\]
we have 
\[
\L(S,\delta)\lesssim\frac{\delta \sqrt{n}}{\gamma}+\sqrt{n}\exp(-\alpha^{2}/2).
\]
\end{proposition}
\begin{proof}
Note that $2S \sim \sum_{i=1}^{n}(1+x_i)a_i$, where $x_i$ are $\{\pm 1\}$-valued random variables sampled uniformly from the hyperplane $x_1+\dots+x_n = 0$, and that $\L(S,\delta) = \L(2S,2\delta) = \L(\sum_{i=1}^{n} x_i a_i,2\delta )$. The desired conclusion follows since for any $r\in \R$,
\begin{align*}
    \Pr\left(\left|\sum_{i=1}^{n}x_i a_i - r\right|\leq 2\delta \right)
    &= \Pr\left(\left|\sum_{i=1}^{n}\epsilon_i a_i - r\right|\leq 2\delta \bigg\vert \epsilon_{1}+\dots+\epsilon_{n}=0 \right)\\
    &\lesssim \sqrt{n}\Pr\left(\left|\sum_{i=1}^{n}\epsilon_i a_i - r\right|\leq 2\delta\right)\\
    &\lesssim \sqrt{n}\L\left(\sum_{i=1}^{n}\epsilon_{i}a_i,2\delta\right)\\
    &\lesssim \frac{\delta \sqrt{n}}{\gamma} + \sqrt{n}\exp(-\alpha^{2}/2),
\end{align*}
where the last inequality follows from \cref{prop:LCD-controls-sbp}.
\end{proof}
The proof of \cref{prop:eliminate-small-LCD-rreg} will be the content of the next two subsections. Here, we present the key initial step, which consists of efficiently passing from vectors on the unit sphere to integer vectors. 
\begin{proposition}
\label{prop:reduction-to-integer-rreg}
With notation as above, we have
\begin{align*}
\Pr\left(\exists \boldsymbol{a} \in \Gamma^{2}(\eta):\|Q_{n}\boldsymbol{a}\|_{2}\leq \eta \right) 
&\lesssim e^{-c_{\ref{prop:boundrestricted-op-norm}}n^{1.02}} + \\ \Pr(\exists \boldsymbol{w} \in (\Z^{n}\setminus\{\boldsymbol{0}\}) \cap [-2\eta^{-1}n^{3/4}, 2\eta^{-1}n^{3/4}]^{n} &: \|Q_{n}\boldsymbol{w}\|_{2}\leq 10C_{\ref{prop:boundrestricted-op-norm}} \min\{n^{0.4}\|\boldsymbol{w}\|_{2},n^{0.9}\}).
\end{align*}
\end{proposition}
\begin{proof}
Since by \cref{prop:boundrestricted-op-norm}, $\Pr\left(\sup_{\boldsymbol{v}\in \boldsymbol{H}\cap \S^{n-1}}\|Q_{n}\boldsymbol{v}\|_{2}\geq C_{\ref{prop:boundrestricted-op-norm}}n^{0.51}\right)\lesssim\exp\left(-c_{\ref{prop:boundrestricted-op-norm}}n^{1.02}\right)$, we may henceforth restrict to the complement of this event. 
Let $\boldsymbol{a}\in \Gamma^{2}(\eta)$. Then, by definition, there exists some $0<\theta \leq \LCD_{\alpha,\gamma}(\boldsymbol{a}) \leq n^{3/4}\eta^{-1}$ and some $\boldsymbol{w}\in \Z^{n}\setminus\{\boldsymbol{0}\}$ such that $\|\theta \boldsymbol{a} - \boldsymbol{w}\|_{2} \leq \min\{\gamma \theta,\alpha\}$.\\ 

\noindent \textbf{Case I: $\gamma\theta\leq n^{-0.6}$}. In particular, $\theta\leq n^{0.4}$.
In this case, if $\|Q_{n}\boldsymbol{a}\|_{2}\leq\eta$, then
\begin{align*}
\|Q_{n}\boldsymbol{w}\|_{2} & =\|Q_{n}(\boldsymbol{w}-\theta\boldsymbol{a})+Q_{n}(\theta\boldsymbol{a})\|_{2}\\
 & \leq\|Q_{n}\|\cdot\|w-\theta\boldsymbol{a}\|_{2}+\theta\cdot\|Q_{n}\boldsymbol{a}\|_{2}\\
 & \leq n\cdot\gamma\theta+\theta\eta\\
 & \leq n^{0.4}+2\eta\|\boldsymbol{w}\|_{2}\\
 & \leq n^{0.4}\|\boldsymbol{w}\|_{2}+2\eta\|\boldsymbol{w}\|_{2}\\
 & \leq3n^{0.4}\|\boldsymbol{w}\|_{2}\\
 & \leq10\min\{n^{0.4}\|\boldsymbol{w}\|_{2},n^{0.8}\}
\end{align*}
where the fourth line uses $\|\boldsymbol{w}\|_{2}\geq\theta(1-\gamma)\geq\theta/2$;
the fifth line uses $\|\boldsymbol{w}\|_{2}\geq1$ (since $\boldsymbol{w}\in\Z^{n}\backslash\{\boldsymbol{0}\}$);
the sixth line uses $\eta\leq n^{0.4}$; and the last line uses $\|\boldsymbol{w}\|_{2}\leq\theta(1+\gamma)\leq2\theta\leq2n^{0.4}$. \\

\noindent \textbf{Case II: $\gamma\theta>n^{-0.6}.$ }In particular, $\theta>n^{-0.6}\gamma^{-1}>n^{0.4}$. Let
\[
\ell:=\left|\langle\boldsymbol{w}-\theta\boldsymbol{a},\boldsymbol{1}\rangle\right|
\]
and let 
\[
s:=\sign(\langle\boldsymbol{w}-\theta\boldsymbol{a},\boldsymbol{1}\rangle).
\]
Let $\boldsymbol{w'}\in\{0,1\}^{n}$ denote the vector whose first $\lfloor\ell\rfloor$
coordinates are $1$ and the remaining coordinates are $0$; note
that this makes sense since, by the Cauchy-Schwarz inequality, we
have 
\[
\ell\leq\|\boldsymbol{w}-\theta\boldsymbol{a}\|_{1}\leq\sqrt{n}\alpha\ll n.
\]
We will need the following easily established claims.
\begin{enumerate}
    \item $|\langle\boldsymbol{w}-s\boldsymbol{w'}-\theta\boldsymbol{a},\boldsymbol{1}\rangle|\leq1$.
Indeed, we have
\begin{align*}
\langle\boldsymbol{w}-s\boldsymbol{w}'-\theta\boldsymbol{a},\boldsymbol{1}\rangle & =\langle\boldsymbol{w}-\theta\boldsymbol{a},\boldsymbol{1}\rangle-s\langle\boldsymbol{w'},\boldsymbol{1}\rangle\\
 & =s\ell-s\lfloor\ell\rfloor\\
 & =s(\ell-\lfloor\ell\rfloor)\\
 & \in[-1,1].
\end{align*}
\item  $\|\boldsymbol{w}-s\boldsymbol{w'}\|_{2}=\theta(1\pm2n^{-1/4})$. This
follows from $\|\boldsymbol{w}-s\boldsymbol{w'}\|_{2}=\|\boldsymbol{w}\|_{2}\pm\|\boldsymbol{w'}\|_{2}$
along with the estimate 
\begin{align*}
\|\boldsymbol{w'}\|_{2}^{2} & \leq\ell\leq\|\boldsymbol{w}-\theta\boldsymbol{a}\|_{1}\leq\sqrt{n}\gamma\theta,
\end{align*}
from which we see that 
\[
\|\boldsymbol{w'}\|_{2}\leq(\sqrt{n}\gamma)^{1/2}\sqrt{\theta}\leq n^{-1/4}\theta.
\]
\item  Restricted to the event $\sup_{\boldsymbol{v}\in \boldsymbol{H}\cap\S^{n-1}}\|Q_{n}\boldsymbol{v}\|_{2}\leq C_{\ref{prop:boundrestricted-op-norm}}n^{0.51}$,
we have
\[
\|Q_{n}(\boldsymbol{w}-s\boldsymbol{w'}-\theta\boldsymbol{a})\|_{2}\leq4C_{\ref{prop:boundrestricted-op-norm}}n^{0.51}\min\{n^{-1/4}\theta,n^{3/8}\}.
\]
Indeed, writing 
\[
\boldsymbol{w}-s\boldsymbol{w'}-\theta\boldsymbol{a}=\langle\boldsymbol{w}-s\boldsymbol{w'}-\theta\boldsymbol{a},\boldsymbol{1}\rangle\frac{\boldsymbol{1}}{n}+\Proj_{\boldsymbol{H}}(\boldsymbol{w}-s\boldsymbol{w'}-\theta\boldsymbol{a}),
\]
we see that 
\begin{align*}
\|Q_{n}(\boldsymbol{w}-s\boldsymbol{w'}-\theta\boldsymbol{a})\|_{2} & \leq\frac{|\langle\boldsymbol{w}-s\boldsymbol{w'}-\theta\boldsymbol{a},\boldsymbol{1}\rangle|}{n}\|Q_{n}\boldsymbol{1}\|_{2}+\|Q_{n}(\Proj_{H}(\boldsymbol{w}-s\boldsymbol{w'}-\theta\boldsymbol{a})\|_{2}\\
 & \leq\frac{1}{n}\cdot n\sqrt{n}+C_{\ref{prop:boundrestricted-op-norm}}n^{0.51}\|\boldsymbol{w}-s\boldsymbol{w'}-\theta\boldsymbol{a}\|_{2}\\
 & \leq\sqrt{n}+C_{\ref{prop:boundrestricted-op-norm}}n^{0.51}\left(\|\boldsymbol{w'}\|_{2}+\|\boldsymbol{w}-\theta\boldsymbol{a}\|_{2}\right),
\end{align*}
where the second inequality uses the estimate from 1. Next, note
that 
\begin{align*}
\|\boldsymbol{w'}\|_{2}+\|\boldsymbol{w}-\theta\boldsymbol{a}\|_{2} & \leq\min\{n^{-1/4}\theta,(\sqrt{n}\alpha)^{1/2}\}+\min\{\gamma\theta,\alpha\}\\
 & \leq2\min\{n^{-1/4}\theta,n^{3/8}\}.
\end{align*}
It follows that 
\begin{align*}
\|Q_{n}(\boldsymbol{w}-s\boldsymbol{w'}-\theta\boldsymbol{a})\|_{2} & \leq\sqrt{n}+2C_{\ref{prop:boundrestricted-op-norm}}n^{0.51}\min\{n^{-1/4}\theta,n^{3/8}\}\\
 & \leq\min\{n^{0.1}\theta,\sqrt{n}\}+2C_{\ref{prop:boundrestricted-op-norm}}n^{0.51}\min\{n^{-1/4}\theta,n^{3/8}\}\\
 & \leq4C_{\ref{prop:boundrestricted-op-norm}}n^{0.51}\min\{n^{-1/4}\theta,n^{3/8}\},
\end{align*}
where the second inequality uses $\theta>n^{0.4}$.
\end{enumerate}

From these facts, it follows that if $\|Q_{n}\boldsymbol{a}\|_{2}\leq\eta$, then  
\begin{align*}
\|Q_{n}(\boldsymbol{w}-s\boldsymbol{w'})\|_{2} & =\|Q_{n}(\boldsymbol{w}-s\boldsymbol{w'}-\theta\boldsymbol{a})+Q_{n}(\theta\boldsymbol{a})\|_{2}\\
 & \leq\|Q_{n}(\boldsymbol{w}-s\boldsymbol{w'}-\theta\boldsymbol{a})\|_{2}+\theta\cdot\|Q_{n}\boldsymbol{a}\|_{2}\\
 & \leq4C_{\ref{prop:boundrestricted-op-norm}}n^{0.51}\min\{n^{-1/4}\theta,n^{3/8}\}+\theta\eta\\
 & \leq4C_{\ref{prop:boundrestricted-op-norm}}n^{0.51}\min\{n^{-1/4}\theta,n^{3/8}\}+n^{3/4}\\
 & \leq4C_{\ref{prop:boundrestricted-op-norm}}n^{0.51}\min\{n^{-1/4}\theta,n^{3/8}\}+n^{0.51}\min\{n^{-1/8}\theta,n^{1/4}\}\\
 & \leq5C_{\ref{prop:boundrestricted-op-norm}}n^{0.51}\min\{n^{-1/8}\theta,n^{3/8}\}\\
 & \leq5C_{\ref{prop:boundrestricted-op-norm}}\min\{n^{0.4}\theta,n^{0.9}\}\\
 & \leq10C_{\ref{prop:boundrestricted-op-norm}}\min\{n^{0.4}\|\boldsymbol{w}-s\boldsymbol{w'}\|_{2},n^{0.9}\}
\end{align*}
where the third line uses 3.; the fourth line uses $\theta\leq n^{3/4}\eta^{-1}$;
the fifth line uses $\theta\geq n^{0.4}$; and the last line uses
2.
\end{proof}
\subsection{Dealing with almost-constant integer vectors}
Throughout this subsection and the next one, $p = 2^{n^{0.001}}$ is a prime. 
\begin{definition}
For an integer vector $\boldsymbol{v}\in \Z^{n}$, we define the size of its largest level set to be
$$L(\boldsymbol{v}) = \sup_{z\in \Z}\left|\{i\in [n] : v_i = z\}\right|.$$
\end{definition}
The goal of this subsection is to prove the following lemma, which follows from \cref{prop:invertibility-single-rreg} and a simple union bound. 
\begin{lemma}
\label{lemma:eliminate-ac-rreg}
$\Pr\left(\exists \boldsymbol{w} \in (\Z^{n}\setminus\{\boldsymbol{0}\})\cap [-p,p]^{n}, L(\boldsymbol{w}) \geq n-n^{0.991}: \|Q_{n}\boldsymbol{w}\|_{2} \leq 10C_{\ref{prop:boundrestricted-op-norm}}n^{0.4}\|\boldsymbol{w}\|_{2}\right) \\ \lesssim \exp(-c_{\ref{lemma:invertibility-fixed-vector}}n/2).$
\end{lemma}
\begin{proof}
The number of vectors $\boldsymbol{w} \in (\Z^{n}\setminus\{\boldsymbol{0}\})\cap [-p,p]^{n}$ with $L(\boldsymbol{w}) \geq n-n^{0.991}$ is at most 
$$\binom{n}{n^{0.991}}\cdot (3p)\cdot (3p)^{n^{0.991}} \ll 2^{n^{0.993}}.$$
For $n \in 2\N$ sufficiently large, by \cref{prop:invertibility-single-rreg}, for any such vector, 
$$\Pr\left(\|Q_{n}\boldsymbol{w}\|_{2} \leq 10C_{\ref{prop:boundrestricted-op-norm}}n^{0.4}\|\boldsymbol{w}\|_{2}\right) \leq \Pr\left(\|Q_{n}\boldsymbol{w}\|_{2} \leq \sqrt{n}\|\boldsymbol{w}\|_{2}/2\right) \leq C_{\ref{prop:invertibility-single-rreg}}(0.006)\exp(-n^{0.994}/4).$$
Therefore, the union bound gives the desired conclusion. 
\end{proof}

\subsection{Dealing with non almost-constant integer vectors}
It remains to deal with integer vectors which are not almost-constant. Formally, let
$$\boldsymbol{V}:=\{\boldsymbol{v}\in(\Z^{n}\setminus\{\boldsymbol{0}\})\cap[-\eta^{-4},\eta^{-4}]^{n}:L(\boldsymbol{v})<n-n^{0.991}\}.
$$
In view of \cref{prop:reduction-to-integer-rreg} and \cref{lemma:eliminate-ac-rreg}, and since $\eta \leq n^{-2}$, the following proposition suffices to prove \cref{prop:eliminate-small-LCD-rreg}.
\begin{proposition}
\label{prop:eliminate-small-LCD-integer-rreg}
$$\Pr\left(\exists \boldsymbol{v}\in \boldsymbol{V} : \|Q_{n} \boldsymbol{v}\|_{2} \leq 10C_{\ref{prop:boundrestricted-op-norm}}n^{0.9} \right) \lesssim 2^{-\sqrt{n}/3}.$$
\end{proposition}
This will be accomplished by a union bound, following the strategy outlined in \cref{eqn:union-bound-over-ball}. However, as compared to the i.i.d. case, there is more work involved. In particular, we will need certain key ideas from \cite{FJLS2018} (where the best-known bounds on the singularity probability of $Q_n$ are obtained), which we now discuss. \\

For $n\in 2\N$, let $\QQ_{n}$ denote the set of all $n\times n$ matrices with entries in $\{0,1\}$, each of whose rows sums to $n/2$. As will soon be clear, we will find it more convenient to work with a `two-step' model for generating a uniformly random element of $\QQ_{n}$. Let $\Sigma_{n}$ denote the set of all permutations on $[n]$, and consider the map
$$f \colon (\Sigma_{n})^{n} \times \left(\{0,1\}^{n/2}\right)^{n} \to \QQ_{n},$$
which takes $\left((\sigma_1,\dots,\sigma_n), \xi_1,\dots,\xi_n\right)$ to the matrix in $\QQ_n$ whose $i^{th}$ row is $(q_{i1},\dots,q_{in})$, where
\[
q_{ij}:=\begin{cases}
\xi_{i}(k) & \text{if }\sigma_{i}(2k-1)=j,\\
1-\xi_{i}(k) & \text{if }\sigma_{i}(2k)=j.
\end{cases}
\]
In other words, for each $k\in [n/2]$, exactly one among the $\sigma_i(2k-1)^{th}$ and $\sigma_i(2k)^{th}$ entries in the $i^{th}$ row is equal to $1$ (the other is equal to $0$), and the value of $\xi_i(k)$ determines which one of the two entries it is. It is straightforward to see that the pushforward measure of the uniform measure on $(\Sigma_n)^{n} \times \left(\{0,1\}^{n/2}\right)^{n}$ under the map $f$ gives the uniform measure on $\QQ_n$. Hence, we have the following process for generating a uniformly random element of $\QQ_n$. First, choose an $n$-tuple of permutations $\boldsymbol{\sigma} = (\sigma_1,\dots,\sigma_n)$, where each coordinate is chosen independently, and uniformly at random from $\Sigma_n$. We shall refer to $\boldsymbol{\sigma}$ as the \emph{base} of the matrix $Q_n$. Second, for each $i\in [n]$ and each $k\in [n/2]$, choose exactly one among the $\sigma_i(2k-1)^{th}$ entry or the $\sigma_i(2k)^{th}$ entry of the $i^{th}$ row of the matrix to be $1$ (and the other to be $0$) uniformly at random, independently for all such values of $i$ and $k$. Let us note here that for each $i\in [n]$, the set comprising the $n/2$ unordered pairs $\{\sigma_i(2k-1),\sigma_i(2k)\}$, for all $k\in [n/2]$, is a uniformly random perfect matching in the complete graph on $n$-vertices $K_n$; we shall refer to this matching as the matching induced by $\sigma_i$. 

As in \cite{FJLS2018}, we will need the notion of an `expanding base', which is formalized in the following definition. 
\begin{definition}
We say that $\boldsymbol{\sigma}:=(\sigma_1,\dots,\sigma_n) \in (\Sigma_n)^{n}$ belongs to $\EE_n$ if it satisfies the following two properties:
\begin{enumerate}[{(Q1)}]
\item The union of any two perfect matchings of the form $\sigma_i$ and $\sigma_j$ ($i\neq j$) has at most $n^{0.6}$ connected components.  
\item For any two subsets $A,B\subseteq [n]$ such that $n^{0.8} \leq |A|,|B|\leq n/2$, there are at most $\sqrt{n}/2$ indices $i\in [n]$ such that the perfect matching induced by $\sigma_i$ has fewer than $|A||B|/(8n)$ edges between $A$ and $B$.   
\end{enumerate}
\end{definition}
It turns out that, with high probability, a uniformly random `base' is `expanding'.
\begin{proposition}[Proposition 5.4 in \cite{FJLS2018}]
\label{prop:expanding-base-perm-whp}
Let $\boldsymbol{\sigma}$ be a uniformly random element of $(\Sigma_n)^{n}$. Then, 
$$\Pr(\boldsymbol{\sigma}\notin \EE_n)\leq 2^{-\sqrt{n}/3}.$$
\end{proposition}

Denote by $Q_{\boldsymbol{\sigma}}$ the random matrix chosen uniformly among all the matrices in $\QQ_{n}$ with base $\boldsymbol{\sigma}$, and by $\boldsymbol{\tau} \in (\Sigma_{n})^{n}$, a vector of i.i.d uniformly random permutations. Then, by the law of total probability, we have
\begin{align*}
\Pr_{Q_{n}}\left(\exists\boldsymbol{v}\in\boldsymbol{V}:\|Q_{n}\boldsymbol{v}\|_{2}\leq10C_{\ref{prop:boundrestricted-op-norm}}n^{0.9}\right) & =\Pr_{Q_{\boldsymbol{\tau}}}\left(\exists\boldsymbol{v}\in\boldsymbol{V}:\|Q_{\boldsymbol{\tau}}\boldsymbol{v}\|_{2}\leq10C_{\ref{prop:boundrestricted-op-norm}}n^{0.9}\right)\\
 & \leq\Pr\left(\exists\boldsymbol{v}\in\boldsymbol{V}:\|Q_{\boldsymbol{\tau}}\boldsymbol{v}\|_{2}\leq10C_{\ref{prop:boundrestricted-op-norm}}n^{0.9}\cap(\tau\in\EE_{n})\right)+\Pr\left(\tau\notin\EE_{n}\right)\\
 & \leq\sup_{\boldsymbol{\sigma}\in\EE_{n}}\Pr\left(\exists\boldsymbol{v}\in\boldsymbol{V}:\|Q_{\boldsymbol{\sigma}}\boldsymbol{v}\|_{2}\leq10C_{\ref{prop:boundrestricted-op-norm}}n^{0.9}\right)+2^{-\sqrt{n}/3},
\end{align*}
where the last inequality is due to \cref{prop:expanding-base-perm-whp}. Thus, in order to prove \cref{prop:eliminate-small-LCD-integer-rreg}, it suffices to prove the following. 
\begin{proposition}
\label{prop:lsv-fixed-sigma}
For any $\boldsymbol{\sigma}\in \EE_{n}$, 
$$\Pr\left(\exists \boldsymbol{v}\in \boldsymbol{V} : \|Q_{\boldsymbol{\sigma}} \boldsymbol{v}\|_{2} \leq 10C_{\ref{prop:boundrestricted-op-norm}}n^{0.9} \right) \lesssim n^{-0.001n}.$$
\end{proposition}
For the remainder of this subsection, fix $\boldsymbol{\sigma}\in \EE_n$. Moreover, fix $k= n^{0.01}$, and $s_1 = s_2 = n^{0.99}$.
For $\boldsymbol{v}\in\boldsymbol{V}$ and $i\in[n]$, we define
$\boldsymbol{v}_{\sigma_{i}}$ to be the $n/2$-dimensional integer vector  
whose $k^{th}$ coordinate is $\left(v_{\sigma_{i}(2k-1)}-v_{\sigma_{i}(2k)}\right)$. This definition is motivated by the following. 
\begin{lemma}
\label{lemma:anti-conc-2-step}
$\sup_{z\in\Z}\Pr\left((Q_{\boldsymbol{\sigma}}\boldsymbol{v})_{i}=z\right) \leq \rho\left(\boldsymbol{v}_{\sigma_{i}}\right).$
\end{lemma}
\begin{proof}
By unwrapping definitions, we see that,
\begin{align*}
\sup_{z\in\Z}\Pr\left((Q_{\boldsymbol{\sigma}}\boldsymbol{v})_{i}=z\right) & =\sup_{z\in\Z}\Pr\left(\sum_{k=1}^{n/2}\frac{v_{\sigma_{i}(2k-1)}+v_{\sigma_{i}(2k)}}{2}+\sum_{k=0}^{n/2}(1-2\xi_{i}(k))\frac{v_{\sigma_{i}(2k-1)}-v_{\sigma_{i}(2k)}}{2}=z\right)\\
 & \leq\sup_{z'\in\Z/2}\Pr\left(\sum_{k=1}^{n/2}(1-2\xi_{i}(k))\frac{v_{\sigma_{i}(2k-1)}-v_{\sigma_{i}(2k)}}{2}=z'\right)\\
 & =\sup_{z\in\Z}\Pr\left(\sum_{k=1}^{n/2}(1-2\xi_{i}(k))\left(v_{\sigma_{i}(2k-1)}-v_{\sigma_{i}(2k)}\right)=z\right)\\
 & =\sup_{z\in\Z}\Pr\left(\sum_{k=1}^{n/2}\epsilon_{i}\boldsymbol{v}_{\sigma_{i}}=z\right)\\
 & =\rho\left(\boldsymbol{v}_{\sigma_{i}}\right).
\end{align*}
\end{proof}
For the purposes of anti-concentration, we would like (as a start) for the vectors $\boldsymbol{v}_{\sigma_i}$ to have sufficiently large support. Accordingly, let 
\[
T_{\boldsymbol{v}}:=\{i\in[n]:|\supp(\boldsymbol{v}_{\sigma_{i}})|\geq n^{0.991}/16\}.
\]
\begin{lemma}
\label{lemma:lower-bound-Tv}
For every $\boldsymbol{v}\in \boldsymbol{V}$, 
$$|T_{\boldsymbol{v}}|\geq n-\sqrt{n}/2.$$
\end{lemma}
\begin{proof}
Note first that for any $\boldsymbol{v}\in \boldsymbol{V}$, the assumption that $L(\boldsymbol{v})<n-n^{0.991}$
implies that there exist disjoint sets $A_{\boldsymbol{v}},B_{\boldsymbol{v}}\subseteq[n]$
such that $|A_{\boldsymbol{v}}|=n^{0.991}$, $|B_{\boldsymbol{v}}|=n/2$
and $v_{i}\neq v_{j}$ for all $i\in A_{\boldsymbol{v}},j\in B_{\boldsymbol{v}}$.
Then, property (Q2) from the definition of $\EE_{n}$ implies that for all
but at most $\sqrt{n}/2$ indices $i\in[n]$, the perfect matching
induced by $\sigma_{i}$ has at least $n^{0.991}/16$ edges with one
endpoint in each of $A_{\boldsymbol{v}}$ and \textbf{$B_{\boldsymbol{v}}$}.
It is easy to see that each such index belongs to $T_{\boldsymbol{v}}$.
\end{proof}

As in the previous section, note that for our choice of parameters, the natural map
$$\iota: \boldsymbol{V} \to \F_{p}^{n}$$
is injective. We will abuse notation, and use $\boldsymbol{v}$ to denote $\iota(\boldsymbol{v})$. This identification allows us to make the next two key definitions, which will enable us to prove effective analogs of \cref{lemma:usable-Halasz} and \cref{corollary:counting} in our setting. 
\begin{definition}[Witnessing pair] For any $\boldsymbol{v}\in \boldsymbol{V}$, we say that the pair $(i_1,i_2) \in T_{\boldsymbol{v}}\times T_{\boldsymbol{v}}$, $i_1 \neq i_2$ witnesses $\boldsymbol{v}$ if 
\[
\min_{\boldsymbol{b}\subseteq\boldsymbol{v}_{\sigma_{i}},|\supp(\boldsymbol{b})|\geq s_{2}}\frac{R_{k}^{*}\left(\boldsymbol{b}\right)}{|\boldsymbol{b}|^{2k}}\geq\min_{\boldsymbol{b}\subseteq\boldsymbol{v}_{\sigma_{2}},|\supp(\boldsymbol{b})|\geq s_{2}}\frac{R_{k}^{*}\left(\boldsymbol{b}\right)}{|\boldsymbol{b}|^{2k}}\geq\max_{i\in T_{\boldsymbol{v}}\backslash\{i_{1},i_{2}\}}\min_{\boldsymbol{b}\subseteq\boldsymbol{v}_{\sigma_{i}},|\supp(\boldsymbol{b})|\geq s_{2}}\frac{R_{k}^{*}\left(\boldsymbol{b}\right)}{|\boldsymbol{b}|^{2k}}.
\]
For a vector $\boldsymbol{v} \in \boldsymbol{V}$, we will denote its witnessing pair (taking the lexicographically first one, in case there are multiple) by $(i_1(\boldsymbol{v}), i_2(\boldsymbol{v}))$. This gives a partition of $\boldsymbol{V}$ into at most $\binom{n}{2}$ parts.  
\end{definition}
\begin{definition} For an integer $t\in [p]$, let 
$$\boldsymbol{V}_{t}:=\left\{ \boldsymbol{v}\in\boldsymbol{V}: \iota\left(\boldsymbol{v}_{\sigma_{i_{2}}(\boldsymbol{v})}\right)\in \Bad_{k,s_2,\geq t-1}^{s_1}(n/2) \setminus \Bad_{k,s_2,\geq t}^{s_1}(n/2) \right\}.$$
\end{definition}
The next two lemmas are the analogs of \cref{corollary:counting} and \cref{lemma:usable-Halasz} respectively in the present setting. 
\begin{lemma}
\label{lemma:counting-rreg}
For our choice of parameters and for any integer $t\in [p]$,
$$|\boldsymbol{V}_{t}| \leq (500)^{n}\left(\frac{p}{t}\right)^{n}.$$
\end{lemma}
\begin{proof}
It is enough to show that the number of vectors $\boldsymbol{v}\in\boldsymbol{V}_{t}$
that are witnessed by a given pair $(i_{1},i_{2})$ of distinct indices
in $T_{\boldsymbol{v}}$ is at most $(400)^{n}(p/t)^{n}$, and then take the union
bound over all such pairs of witnessing indices. Let us now fix such
a pair for the remainder of the proof. 

It follows from the definition of a witnessing sequence that both
$\iota(\boldsymbol{v}_{\sigma_{i_{1}}})$ and $\iota(\boldsymbol{v}_{\sigma_{i_{2}}})$
belong to $\Bad_{k,s_{2},\geq t-1}^{s_{1}}(n/2)$. Hence, \cref{theorem:counting} shows that
each of the vectors $\iota(\boldsymbol{v}_{\sigma_{i_{1}}})$ and $\iota(\boldsymbol{v}_{\sigma_{i_{2}}})$
belong to a set of size at most 
\[
(300)^{n/2}\left(\frac{p}{t}\right)^{n/2},
\]
and the injectivity of $\iota$ gives the same conclusion for $\boldsymbol{v}_{\sigma_{i_{1}}}$ and $\boldsymbol{v}_{\sigma_{i_{2}}}$.

Next, we bound the number of vectors $\boldsymbol{v}\in\boldsymbol{V}$
with a given value of $\left(\boldsymbol{v}_{\sigma_{i_{1}}},\boldsymbol{v}_{\sigma_{i_{2}}}\right)$.
Note that all such vectors $\boldsymbol{v}$ have the same differences
between all those pairs of coordinates that are connected by an edge
of the union of the matchings induced by $\sigma_{i_{1}}$ and $\sigma_{i_{2}}$.
In particular, each vector $\boldsymbol{v}$ is uniquely determined
once we fix the value of a single coordinate in each connected component
of this graph. Since property (Q1) from the definition of $\EE_{n}$
implies that the number of connected components does not exceed $n^{0.6}$,
we may conclude that 
\[
\left|\boldsymbol{V}_{t}\right|\leq p^{n^{0.6}}\cdot\left((300)^{n/2}(p/t)^{n/2}\right)^{2}\leq(400)^{n}\left(\frac{p}{t}\right)^{n}.
\]
\end{proof}
\begin{lemma}
\label{lemma:usable-halasz-rreg}
There exists an absolute constant $C_{\ref{lemma:usable-halasz-rreg}}$ such that for our choice of parameters and for any integer $t\in [p]$, if $\boldsymbol{v}\in \boldsymbol{V}_{t}$, then
\begin{itemize}
    \item For any $i \in T_{\boldsymbol{v}}\setminus {i_1(\boldsymbol{v})}$,
    $$\sup_{z\in \Z}\Pr\left((Q_{\boldsymbol{\sigma}}\boldsymbol{v})_{i}=z\right) \leq \left(\frac{C_{\ref{lemma:usable-halasz-rreg}}}{p}\left(\frac{t}{n^{0.48}}+1\right)\right).$$
    \item For any $\boldsymbol{z}\in \Z^{n}$, 
    $$\Pr\left(Q_{\boldsymbol{\sigma}}\boldsymbol{v} = \boldsymbol{z}\right) \leq \left(\frac{C_{\ref{lemma:usable-halasz-rreg}}}{p}\left(\frac{t}{n^{0.48}}+1\right)\right)^{n-\sqrt{n}}.$$
\end{itemize}
\end{lemma}
\begin{proof}
It follows from the definition of a witnessing pair that, for
each $\boldsymbol{v}\in\boldsymbol{V}_{t}$ and for every $i\in T_{\boldsymbol{v}}\setminus\{i_{1}(\boldsymbol{v})\}$,
we have 
\[
\min_{\boldsymbol{b}\subseteq\boldsymbol{v}_{\sigma_{i}},|\supp(\boldsymbol{b})|\geq s_{2}}\frac{R_{k}^{*}(\boldsymbol{b})}{|\boldsymbol{b}|^{2k}}\leq\min_{\boldsymbol{b}\subseteq\boldsymbol{v}_{\sigma_{i_{2}(\boldsymbol{v})}},|\supp(\boldsymbol{b})|\geq s_{2}}\frac{R_{k}^{*}(\boldsymbol{b})}{|\boldsymbol{b}|^{2k}}<\frac{t\cdot2^{2k}}{p}.
\]
In particular, for all $i\in T_{\boldsymbol{v}}\backslash\{i_{1}(\boldsymbol{v})\}$,
$\boldsymbol{v}_{\sigma_{i}}\notin \Bad_{k,s_{2},\geq(t+1)}^{s_{1}}(n/2)$.
Therefore, by essentially the same computation as in \cref{lemma:usable-Halasz}, we have
for all $i\in T_{\boldsymbol{v}}\backslash\{i_{1}(\boldsymbol{v})\}$ that 
\[
\rho\left(\boldsymbol{v}_{\sigma_{i}}\right)\leq\frac{C_{\ref{lemma:usable-Halasz}}}{p}\left(\frac{t}{n^{0.48}}+1\right),
\]
so that the first bullet point follows from \cref{lemma:anti-conc-2-step}. The second bullet point follows immediately, since for any $\boldsymbol{z}\in \Z^{n}$, 
so that for any $\boldsymbol{z}\in\Z^{n}$, 
\begin{align*}
\Pr\left(Q_{\boldsymbol{\sigma}}\boldsymbol{v}=\boldsymbol{z}\right) & \leq\Pr\left(\left(Q_{\boldsymbol{\sigma}}\boldsymbol{v}\right)_{i}=z_{i}\forall i\in T_{\boldsymbol{v}}\backslash\{i_{i}(\boldsymbol{v})\}\right)\\
 & \leq\left(\frac{C_{\ref{lemma:usable-Halasz}}}{p}\left(\frac{t}{n^{0.48}}+1\right)\right)^{|T_{\boldsymbol{v}}|-1}\\
 & \leq\left(\frac{C_{\ref{lemma:usable-Halasz}}}{p}\left(\frac{t}{n^{0.48}}+1\right)\right)^{n-\sqrt{n}},
\end{align*}
where the final inequality follows from \cref{lemma:lower-bound-Tv}. 
\end{proof}
Finally, we are in a position to prove \cref{prop:lsv-fixed-sigma}. As discussed earlier, this completes the proof of \cref{prop:eliminate-small-LCD-rreg}, and hence, the proof of \cref{thm:lsv-rreg}.
\begin{proof}[Proof of \cref{prop:lsv-fixed-sigma}]
We begin by noting that every $\boldsymbol{v}\in\boldsymbol{V}$ satisfies
\[
\sup_{z\in\Z}\Pr\left((Q_{\boldsymbol{\sigma}}\boldsymbol{v})_{i}=z\right)\geq \eta^{4}n^{-1}/3
\]
by the same pigeonhole argument as in the proof of \cref{prop:eliminate-smallLCD-integer}. Since $\eta^{4}n^{-1}\gg 1/\sqrt{p}$,
it follows from \cref{lemma:usable-halasz-rreg} that $\boldsymbol{V}_{t}=\emptyset$ for all
$t\leq\sqrt{p}$. 

On the other hand, using \cref{eqn:union-bound-over-ball} with $\Gamma=\boldsymbol{V}_{t}$ and
$C(n)= 10C_{\ref{prop:boundrestricted-op-norm}}n^{0.4}$, it follows from \cref{lemma:counting-rreg} and \cref{lemma:usable-halasz-rreg} that for all $t\geq\sqrt{p}$,
the probability that the image of any vector in $\boldsymbol{V}_{t}$
under $Q_{\boldsymbol{\sigma}}$ lies in the ball of radius $10C_{\ref{prop:boundrestricted-op-norm}}n^{0.9}$ centered at the origin is
at most 
\begin{align*}
(1000C_{\ref{prop:boundrestricted-op-norm}}n^{0.4})^{n}|\boldsymbol{V}_{t}|\left(\frac{2C_{\ref{lemma:usable-halasz-rreg}}t}{pn^{0.48}}\right)^{n-\sqrt{n}}
&\leq(500000C_{\ref{prop:boundrestricted-op-norm}}n^{0.4})^{n}\left(\frac{p}{t}\right)^{n}\left(\frac{2C_{\ref{lemma:usable-halasz-rreg}}t}{pn^{0.48}}\right)^{n-\sqrt{n}}\\
&\leq(500000C_{\ref{prop:boundrestricted-op-norm}})^{n}\left(\frac{pn}{t}\right)^{\sqrt{n}}\left(\frac{2C_{\ref{lemma:usable-halasz-rreg}}}{n^{0.08}}\right)^{{n}}\\
& \ll n^{-0.01n}.
\end{align*} 
Finally, taking the union bound over integers $t\in [\sqrt{p},p]$ completes the proof. 
\end{proof}
\bibliographystyle{abbrv}
\bibliography{least-singular-value}

\begin{thebibliography}{10}

\bibitem{bai1988note}
Z.~D. Bai, J.~W. Silverstein, and Y.~Q. Yin.
\newblock A note on the largest eigenvalue of a large dimensional sample
  covariance matrix.
\newblock {\em Journal of Multivariate Analysis}, 26(2):166--168, 1988.

\bibitem{cook2017circular}
N.~A. Cook.
\newblock The circular law for random regular digraphs.
\newblock {\em arXiv preprint arXiv:1703.05839}, 2017.

\bibitem{cook2017singularity}
N.~A. Cook.
\newblock On the singularity of adjacency matrices for random regular digraphs.
\newblock {\em Probability Theory and Related Fields}, 167(1-2):143--200, 2017.

\bibitem{costello2006random}
K.~P. Costello, T.~Tao, and V.~H. Vu.
\newblock Random symmetric matrices are almost surely nonsingular.
\newblock {\em Duke Mathematical Journal}, 135(2):395--413, 2006.

\bibitem{edelman1988eigenvalues}
A.~Edelman.
\newblock Eigenvalues and condition numbers of random matrices.
\newblock {\em SIAM Journal on Matrix Analysis and Applications},
  9(4):543--560, 1988.

\bibitem{esseen1966kolmogorov}
C.~Esseen.
\newblock On the {K}olmogorov-{R}ogozin inequality for the concentration
  function.
\newblock {\em Probability Theory and Related Fields}, 5(3):210--216, 1966.

\bibitem{FJLS2018}
A.~Ferber, V.~Jain, K.~Luh, and W.~Samotij.
\newblock On the counting problem in inverse {L}ittlewood--{O}fford theory.
\newblock {\em arXiv preprint arXiv:1904.10425}, 2019.

\bibitem{huang2018invertibility}
J.~Huang.
\newblock Invertibility of adjacency matrices for random $d$-regular graphs.
\newblock {\em arXiv preprint arXiv:1807.06465}, 2018.

\bibitem{jain2019b}
V.~Jain.
\newblock Approximate {S}pielman-{T}eng theorems for random matrices with heavy
  tailed entries: a combinatorial view.
\newblock {\em In preparation}, 2019.

\bibitem{jain2019c}
V.~Jain.
\newblock Smoothed analysis of the condition number without inverse
  {L}ittlewood-{O}fford theory.
\newblock {\em In preparation}, 2019.

\bibitem{kahn1995probability}
J.~Kahn, J.~Koml{\'o}s, and E.~Szemer{\'e}di.
\newblock On the probability that a random $\pm$1-matrix is singular.
\newblock {\em Journal of the American Mathematical Society}, 8(1):223--240,
  1995.

\bibitem{komlos1967determinant}
J.~Koml{\'o}s.
\newblock On determinant of (0, 1) matrices.
\newblock {\em Studia Science Mathematics Hungarica}, 2:7--21, 1967.

\bibitem{kwan2018anticoncentration}
M.~Kwan, B.~Sudakov, and T.~Tran.
\newblock Anticoncentration for subgraph statistics.
\newblock {\em Journal of the London Mathematical Society}, 2018.

\bibitem{landon2019fixed}
B.~Landon, P.~Sosoe, and H.-T. Yau.
\newblock Fixed energy universality of {D}yson {B}rownian motion.
\newblock {\em Advances in Mathematics}, 346:1137--1332, 2019.

\bibitem{latala2005some}
R.~Lata{\l}a.
\newblock Some estimates of norms of random matrices.
\newblock {\em Proceedings of the American Mathematical Society},
  133(5):1273--1282, 2005.

\bibitem{litvak2017adjacency}
A.~E. Litvak, A.~Lytova, K.~Tikhomirov, N.~Tomczak-Jaegermann, and P.~Youssef.
\newblock Adjacency matrices of random digraphs: singularity and
  anti-concentration.
\newblock {\em Journal of Mathematical Analysis and Applications},
  445(2):1447--1491, 2017.

\bibitem{litvak2017smallest}
A.~E. Litvak, A.~Lytova, K.~Tikhomirov, N.~Tomczak-Jaegermann, and P.~Youssef.
\newblock The smallest singular value of a shifted d-regular random square
  matrix.
\newblock {\em Probability Theory and Related Fields}, pages 1--47, 2017.

\bibitem{litvak2005smallest}
A.~E. Litvak, A.~Pajor, M.~Rudelson, and N.~Tomczak-Jaegermann.
\newblock Smallest singular value of random matrices and geometry of random
  polytopes.
\newblock {\em Advances in Mathematics}, 195(2):491--523, 2005.

\bibitem{meszaros2018distribution}
A.~M{\'e}sz{\'a}ros.
\newblock The distribution of sandpile groups of random regular graphs.
\newblock {\em arXiv preprint arXiv:1806.03736}, 2018.

\bibitem{nguyen2013singularity}
H.~H. Nguyen.
\newblock On the singularity of random combinatorial matrices.
\newblock {\em SIAM Journal on Discrete Mathematics}, 27(1):447--458, 2013.

\bibitem{nguyen2012circular}
H.~H. Nguyen and V.~Vu.
\newblock Circular law for random discrete matrices of given row sum.
\newblock {\em arXiv preprint arXiv:1203.5941}, 2012.

\bibitem{nguyen2013small}
H.~H. Nguyen and V.~H. Vu.
\newblock Small ball probability, inverse theorems, and applications.
\newblock In {\em Erd{\H{o}}s Centennial}, pages 409--463. Springer, 2013.

\bibitem{nguyen2018nonsingularity}
H.~H. Nguyen and M.~M. Wood.
\newblock Nonsingularity of adjacency matrices of random $r$-regular graphs.
\newblock {\em arXiv preprint arXiv:1806.10068}, 2018.

\bibitem{o2014analysis}
R.~O'Donnell.
\newblock {\em Analysis of {B}oolean functions}.
\newblock Cambridge University Press, 2014.

\bibitem{rebrova2018coverings}
E.~Rebrova and K.~Tikhomirov.
\newblock Coverings of random ellipsoids, and invertibility of matrices with
  iid heavy-tailed entries.
\newblock {\em Israel Journal of Mathematics}, 227(2):507--544, 2018.

\bibitem{rudelson2008invertibility}
M.~Rudelson.
\newblock Invertibility of random matrices: norm of the inverse.
\newblock {\em Annals of Mathematics}, pages 575--600, 2008.

\bibitem{rudelson2013lecture}
M.~Rudelson.
\newblock Lecture notes on non-asymptotic theory of random matrices.
\newblock 2013.

\bibitem{rudelson2008littlewood}
M.~Rudelson and R.~Vershynin.
\newblock The {L}ittlewood--{O}fford problem and invertibility of random
  matrices.
\newblock {\em Advances in Mathematics}, 218(2):600--633, 2008.

\bibitem{rudelson2010non}
M.~Rudelson and R.~Vershynin.
\newblock Non-asymptotic theory of random matrices: extreme singular values.
\newblock In {\em Proceedings of the International Congress of Mathematicians
  2010 (ICM 2010) (In 4 Volumes) Vol. I: Plenary Lectures and Ceremonies Vols.
  II--IV: Invited Lectures}, pages 1576--1602. World Scientific, 2010.

\bibitem{spielman2004smoothed}
D.~A. Spielman and S.-H. Teng.
\newblock Smoothed analysis of algorithms: {W}hy the simplex algorithm usually
  takes polynomial time.
\newblock {\em Journal of the ACM (JACM)}, 51(3):385--463, 2004.

\bibitem{tao2012topics}
T.~Tao.
\newblock {\em Topics in random matrix theory}, volume 132.
\newblock American Mathematical Soc., 2012.

\bibitem{tao2008random}
T.~Tao and V.~Vu.
\newblock Random matrices: the circular law.
\newblock {\em Communications in Contemporary Mathematics}, 10(02):261--307,
  2008.

\bibitem{tao2006additive}
T.~Tao and V.~H. Vu.
\newblock {\em Additive combinatorics}, volume 105.
\newblock Cambridge University Press, 2006.

\bibitem{tao2007singularity}
T.~Tao and V.~H. Vu.
\newblock On the singularity probability of random {B}ernoulli matrices.
\newblock {\em Journal of the American Mathematical Society}, 20(3):603--628,
  2007.

\bibitem{tao2009inverse}
T.~Tao and V.~H. Vu.
\newblock Inverse {L}ittlewood-{O}fford theorems and the condition number of
  random discrete matrices.
\newblock {\em Annals of Mathematics}, pages 595--632, 2009.

\bibitem{tikhomirov2018singularity}
K.~Tikhomirov.
\newblock Singularity of random {B}ernoulli matrices.
\newblock {\em arXiv preprint arXiv:1812.09016}, 2018.

\bibitem{vershynin2010introduction}
R.~Vershynin.
\newblock Introduction to the non-asymptotic analysis of random matrices.
\newblock {\em arXiv preprint arXiv:1011.3027}, 2010.

\bibitem{vershynin2014invertibility}
R.~Vershynin.
\newblock Invertibility of symmetric random matrices.
\newblock {\em Random Structures \& Algorithms}, 44(2):135--182, 2014.

\bibitem{yin1988limit}
Y.-Q. Yin, Z.-D. Bai, and P.~R. Krishnaiah.
\newblock On the limit of the largest eigenvalue of the large dimensional
  sample covariance matrix.
\newblock {\em Probability theory and related fields}, 78(4):509--521, 1988.

\end{thebibliography}
\end{document}